\documentclass[12pt]{article}
\usepackage{amsfonts}
\usepackage{amsmath,amssymb,amsfonts,amsthm}
\usepackage[usenames]{color}

\usepackage[utf8]{inputenc}

\usepackage{url}

\theoremstyle{definition}

\newtheorem{definition}{Definition}
\newtheorem{theorem}{Theorem}

\newtheorem{remark}{Remark}
\newtheorem{proposition}{Proposition}
\newtheorem{corollary}{Corollary}

\newcommand{\N}{\mathbb{N}}
\newcommand{\C}{\mathbb{C}}
\newcommand{\R}{\mathbb{R}}
\newcommand{\Z}{\mathbb{Z}}

\newcommand{\HH}{\mathbb{H}}

\newcommand{\Ad}{\operatorname{Ad}}
\newcommand{\Ker}{\operatorname{Ker}}
\newcommand{\Sp}{\operatorname{Sp}}
\newcommand{\SU}{\operatorname{SU}}
\newcommand{\SO}{\operatorname{SO}}
\newcommand{\UU}{\operatorname{U}}
\newcommand{\Spin}{\operatorname{Spin}}
\newcommand{\Sphere}{\mathbb{S}}
\newcommand{\Toro}{\operatorname{T}}

\title{\textsc{Matrix spherical analysis on nilmanifolds}}
\author{Rocío Díaz Martín and Linda Saal} 
\date{}

\begin{document} 
\maketitle

\textsc{Abstract.} 
Given a nilpotent Lie group $N$, a compact subgroup $K$ of automorphisms of $N$ and an irreducible unitary representation $(\tau,W_\tau)$ of $K$, we study conditions on $\tau$ for the commutativity of the algebra of $\mathrm{End}(W_\tau)$-valued integrable functions on $N$, with an additional property that generalizes the notion of $K$-invariance. A necessary condition, proved by F. Ricci and A. Samanta, is that $(K\ltimes N,K)$ must be a Gelfand pair. In this article we determine all the commutative algebras from a particular class of Gelfand pairs constructed by J. Lauret. 
\\


\pagestyle{headings}
\pagenumbering{arabic}

\section{Introduction}

\quad Let $N$ be a connected  and simply connected nilpotent Lie group endowed with a left-invariant Riemannian metric determined by an inner product $\langle\cdot,\cdot\rangle$ on its Lie algebra $\mathfrak{n}$. The Riemannian manifold $(N,\langle\cdot,\cdot\rangle)$ is said to be a \textit{homogeneous nilmanifold} (cf. \cite{Lauret nil}). Its isometry group  is given by the semidirect product $K\ltimes N$, where $K$ is  the group of orthogonal automorphisms of $N$ (cf. \cite{Wil}). The action of $K$ on $N$ will be denoted by $k\cdot x$.  We say that $(K,N)$ (or $(K\ltimes N,K)$) is a \textit{Gelfand pair} when the convolution algebra of $K$-invariant integrable functions on $N$ is commutative. Let $(\tau, W_\tau)$ be an 
 irreducible unitary representation of $K$. We say that  $(K,N,\tau)$ is a \textit{commutative triple} when the convolution algebra of $\mathrm{End}(W_\tau)$-valued integrable functions on $N$ such that 
$F(k\cdot x)=\tau(k)F(x)\tau(k)^{-1}$ (for all $k\in K$  and  $x\in N$) is commutative. Finally, we say that $(K,N)$ is a \textit{strong Gelfand pair} if $(K,N,\tau)$ is commutative for every $\tau\in\widehat{K}$ (where $\widehat{K}$ denotes the set of equivalence classes of 
irreducible unitary representations of $K$).
It is shown in \cite{Fulvio} that if there exists $\tau \in\widehat{K}$ such that $(K,N,\tau)$ is a commutative triple, then $(K,N)$ is a Gelfand pair. 
\\

Starting from a faithful real representation $V$ of a compact Lie algebra $\mathfrak{g}$, J. Lauret  constructs in  \cite{Lauret} a two-step nilpotent Lie group $N(\mathfrak{g},V)$  and  he gives a full classification of the Gelfand pairs $(K,N(\mathfrak{g},V))$. The aim of this article is to determine the nontrivial commutative triples occurring for the indecomposable Gelfand pairs given there. We will restrict our attention to the groups $N(\mathfrak{g},V)$ having square integrable representations. In these cases we will see that it is possible to reduce the problem to studying the commutativity of triples $(K',H,\tau_{|_{K'}})$, where $H$ is a Heisenberg group and $K'$ is a subgroup of the orthogonal automorphisms of $H$ contained in $K$ and $\tau_{|_{K'}}$ denotes the restriction of $\tau$ to $K'$. 
\\

The main goal of this paper will be to prove the following theorem.
\begin{theorem}\label{teo 1}
Let $(K,N(\mathfrak{g},V))$ be an indecomposable Gelfand pair such that $N(\mathfrak{g},V)$ has a square integrable representation. The complete list of commutative triples $(K,N(\mathfrak{g},V),\tau)$ with $\tau\in \widehat{K}$ is the following:

\begin{itemize}

\item $\left(\SU(2)\times \Sp(n), N(\mathfrak{su}(2),(\mathbb{C}^2)^n), \tau \right), \ n\geq 1$ (Heisenberg-type), where $\mathbb{C}^2$ denotes the standard representation of $\mathfrak{su}(2)$ regarded as a real representation  (and $\mathfrak{su}(2)$ acts on $(\C^2)^n$ component-wise) and $\tau\in\widehat{\SU(2)}$ or $\tau\in\widehat{\Sp(n)}$ with highest weight  associated to a constant partition of length at most $n$.

\item $(\SU(n)\times \Sphere^1, N(\mathfrak{su}(n),\mathbb{C}^n), \tau), \ n\geq 3$, where $\mathbb{C}^n$ denotes the standard representation of $\mathfrak{su}(n)$ regarded as a real representation and $\tau$ is a character of $\Sphere^1$.

\item $(\SU(n)\times \Sphere^1, N(\mathfrak{u}(n),\mathbb{C}^n), \tau), \ n\geq 3$, where $\mathbb{C}^n$ denotes the standard representation of $\mathfrak{u}(n)$ regarded as a real representation and $\tau$ is a character of $\Sphere^1$.

\item $(\SU(2)\times \UU(k)\times \Sp(n), N(\mathfrak{u}(2),(\mathbb{C}^2)^k\oplus(\mathbb{C}^2)^n),\tau), \ k\geq 1, n\geq 0$, where the center of $\mathfrak{u}(2)$ acts non-trivially only on $(\mathbb{C}^2)^k$, $(\mathbb{C}^2)^n$ denotes the representation of $\mathfrak{su}(2)$ stated in the first item  and $\mathfrak{u}(2)$ acts component-wise on $(\mathbb{C}^2)^k$ in the standard way regarded as a real representation and $\tau\in \widehat{\UU(k)}$.

\item $(K,N(\mathfrak{g},V),\tau)$ where:
\begin{itemize}
\item $\mathfrak{g}:=\mathfrak{su}(m_1)\oplus...\oplus\mathfrak{su}(m_\beta)\oplus\mathfrak{su}(2)\oplus...\oplus\mathfrak{su}(2)\oplus \mathfrak{c}$, where  there are $\alpha$  copies of $\mathfrak{su}(2)$, $m_i\geq 3$ for all $1\leq i\leq \beta$ and $\mathfrak{c}$ is an abelian component.
\item $V:=\C^{m_1}\oplus...\oplus \C^{m_\beta}\oplus\C^{2k_1+2n_1}\oplus...\oplus\C^{2k_\alpha+2n_\alpha}$, where $k_j\geq 1$ and $n_j\geq 0$ for all $1\leq j\leq \alpha$.
\item $\mathfrak{g}$ is acting on $V$ as follows: For each $1\leq i\leq \beta+\alpha$, $\mathfrak{c}$ has a maximal subspace, denoted by $\mathfrak{c}_i$ and with $\dim(\mathfrak{c}_i)=1$,  acting non-trivially only on one component of $V$. For $1\leq i\leq \beta$,
 $\mathfrak{su}(m_i)\oplus \mathfrak{c}_i$  acts non-trivially only on  $\C^{m_i}$  and for $\beta+1<i\leq \beta+\alpha $, $\mathfrak{su}(2)\oplus\mathfrak{c}_i$  acts non-trivially only on $\C^{2k_i+2n_i}$.
\item $K=G\times U$ where 
\begin{itemize}
\item $G:= \SU(m_1)\times...\SU(m_\beta)\times \SU(2)\times...\times \SU(2)$, with $\alpha$ copies of  $\SU(2)$ and 
\item $U:=\Sphere^1\times...\times \Sphere^1\times \UU(k_1)\times \Sp(n_1)\times...\times \UU(k_\alpha)\times \Sp(n_\alpha)$, with $\beta$ copies of  $\Sphere^1$.
\end{itemize}
\item $\tau\in\widehat{\Sphere^1}\otimes...\otimes \widehat{\Sphere^1}\otimes  \widehat{\UU(k_1)}\otimes...\otimes \widehat{\UU(k_\alpha)}$.
\end{itemize}

\item $(\UU(n), N(\R,\mathbb{C}^n)), \ n\geq 1$ (Heisenberg group) is a strong Gelfand pair (proved in \cite{Yakimova}).
\end{itemize}
\end{theorem}

\textsc{Acknowledgments.} 
We are deeply grateful to Leandro Cagliero for many conversations concerning on decomposition of tensor products of representations.
We also thank the referees whose observations improved our first version of this paper significantly. 


\section{Some preliminary results}\label{section prelim}


\quad  Let $N$ be a connected and simply connected nilpotent Lie endowed with a left-invariant Riemannian metric determined by an inner product $\langle\cdot,\cdot\rangle$ on its Lie algebra.  We denote by $\widehat{N}$ the set of equivalence classes of  irreducible unitary representations of $N$.
The Plancherel theorem states that 
there exists a measure $\mu$ on $\widehat{N}$ such that 
\begin{equation*}
\displaystyle{\int_N |f(x)|^2 dx=\int_{\widehat{N}}\|\rho(f)\|_{HS}^2 \  d\mu(\rho)},
\end{equation*}
for all $f\in L^1(N)\cap L^2(N)$, where $dx$ is the Haar measure on $N$, $HS$ stands for the Hilbert-Schmidt norm, and for $\rho\in\widehat{N}$, $\rho(f)$ is the operator defined by $\rho(f):=\int_N f(x)\rho(x)dx$.
It has a natural extension to 
matrix-valued functions as follows. Let $M_m(\mathbb{C})$ denote the set of all square $m\times m$ matrices over $\mathbb{C}$.

\begin{theorem}\label{Plancherel}
Let $F\in L^1(N, M_m(\mathbb{C}))\cap L^2(N, M_m(\mathbb{C}))$. 
Then
\begin{equation*}
\|F\|_2^2=\int_{\widehat{N}} \|\rho(F)\|_{HS}^2 \ d\mu(\rho),
\end{equation*}
where $\rho(F)$ is the operator defined by $\rho(F):=\int_N \rho(x)\otimes F(x) dx$.
\end{theorem}
\begin{proof}
Let $\{e_i \}_{i=1}^m$ be a basis of $\mathbb{C}^m$.
The classical Plancherel theorem holds for all the matrix entries $F_{i,j}(x):=\langle F(x)e_j,e_i\rangle$ of $F$. Therefore
\begin{align*}
\|F\|_2^2 &:=\sum_{i,j=1}^m \|F_{i,j}\|_2^2 \\
&=\sum_{i,j=1}^m \int_{\widehat{N}} \|\rho(F_{i,j})\|_{HS}^2  \ d\mu(\rho).
\end{align*}
For each $(\rho,H_\rho)\in\widehat{N}$, let $\{h_\alpha\}$ be an orthonormal basis of the Hilbert space $H_\rho$, whence
\begin{align*}
\sum_{i,j=1}^m \int_{\widehat{N}} \|\rho(F_{i,j})\|_{HS}^2  \ d\mu(\rho)&=
\sum_{i,j=1}^m\int_{\widehat{N}} \sum_{\alpha,\beta}
|\langle\rho(F_{i,j})h_\alpha,h_\beta\rangle|^2 \ d\mu(\rho) \\
&=\sum_{i,j=1}^m\int_{\widehat{N}} \sum_{\alpha,\beta}
|\int_N \langle\rho(x) h_\alpha, h_\beta\rangle \langle F(x) e_j,e_i\rangle dx|^2 \ d\mu(\rho)\\
&=\int_{\widehat{N}}\sum_{i,j=1}^m \sum_{\alpha,\beta}
|\int_N \langle \rho(x)\otimes F(x) (h_\alpha\otimes e_j), h_\beta\otimes e_i\rangle dx|^2 \ d\mu(\rho)\\
&=\int_{\widehat{N}}\sum_{i,j=1}^m \sum_{\alpha,\beta}
|\langle\rho(F) (h_\alpha\otimes e_j), h_\beta\otimes e_i\rangle |^2 \ d\mu(\rho)\\
&=\int_{\widehat{N}} \|\rho(F)\|
_{HS}^2 \ d\mu(\rho).
\end{align*}
\end{proof}
Let $Z$ be the center of $N$. A representation  $ (\rho, H_\rho)\in\widehat{N}$ is said to be \textit{square integrable} if its matrix coefficients $\langle u,\rho(x)v\rangle$, for $u,v\in H_\rho$,  are square integrable functions on $N$ module $Z$. 
We denote by $\widehat{N_\mathrm{sq}}$ the subset of $\widehat{N}$ of square integrable classes. 
Theorem 14.2.14 in \cite{Wolf} states that if $N$ has a square integrable representation, its Plancherel measure is concentrated on $\widehat{N_\mathrm{sq}}$. 
\\

Let $K$ be a compact subgroup of automorphisms of $N$.
On the one hand it is shown in \cite{BJR} that if $(K,N)$ is a Gelfand pair, then $N$ must be abelian or two-step nilpotent. On the other hand, in \cite{Fulvio} it is proved that if there exists $\tau \in\widehat{K}$ such that $(K,N,\tau)$ is a commutative triple, then $(K,N)$ must be a Gelfand pair. For these reasons we deal only with two-step nilpotent groups $N$. Moreover, 
we assume that $N$  has a square integrable representation.
\\

Let $\mathfrak{n}$ be the Lie algebra of $N$ with Lie bracket  $[\cdot,\cdot]$. Consequently, $\mathfrak{n}=\mathfrak{z}\oplus V$ where $\mathfrak{z}$ is its center and $V$ is the orthogonal complement of $\mathfrak{z}$. The group $N$ acts naturally on its Lie algebra $\mathfrak{n}$  by  the adjoint action $\Ad$. Also, $N$ acts on $\mathfrak{n}^*$, the real dual space of $\mathfrak{n}$, by the dual or contragredient representation of the adjoint representation
$\mathrm{Ad}^*(n)\lambda:=\lambda\circ \Ad(n^{-1})$  (for all $n\in N$ and  $\lambda\in\mathfrak{n}^*$). 
Fixed $\lambda\in\mathfrak{n}^*$ nontrivial, let $O(\lambda):=\{\mathrm{Ad}^*(n)\lambda | \  n\in N\}$ be its coadjoint orbit. 
\\

From Kirillov's theory there is a correspondence between $\widehat{N}$ and the set of coadjoint orbits. Let  $B_\lambda$ be the skew symmetric bilinear form on $\mathfrak{n}$ given by $$B_\lambda(X,Y):=\lambda([X,Y]) \quad (X,Y\in \mathfrak{n}).$$ 
Let
$\mathfrak{m}\subset\mathfrak{n}$ be a maximal isotropic subalgebra in the sense that $B_\lambda(X,Y)=0$ for all $X,Y\in \mathfrak{m}$ and let $M:=\exp(\mathfrak{m})$. 
Defining on $M$ the character  $\chi_\lambda(\exp(Y)):=e^{i\lambda(Y)}$  (for all $Y \in\mathfrak{m}$), the irreducible 
representation  $\rho_\lambda\in\widehat{N}$ associated to $O(\lambda)$ is the induced representation $\rho_\lambda:=\mathrm{Ind}_M^N(\chi_\lambda)$. 
\\

Let $X_\lambda\in\mathfrak{z}$ be the representative of $\lambda_{|_\mathfrak{z}}$ (the restriction of $\lambda$ to $\mathfrak{z}$), that is,  $\lambda(Y)=\langle Y,X_\lambda\rangle$ for all $Y\in \mathfrak{z}$. We can split  $\mathfrak{z}=\mathbb{R}X_\lambda\oplus \mathfrak{z}_\lambda$, where $\mathfrak{z}_\lambda:=\Ker(\lambda_{|_{\mathfrak{z}}})$ is the orthogonal complement of $\mathbb{R}X_\lambda$ on $\mathfrak{z}$.
\\

One immediately sees that $\rho_\lambda\in \widehat{N_\mathrm{sq}}$ implies that $B_\lambda$ is non-degenerate on $V$ and that the orbits are maximal, i.e., they are of the form $O(\lambda)=\lambda\oplus V^*$. Indeed, let $\mathfrak{a}_\lambda$ be the subspace of $V$ where $B_\lambda$ is degenerate, i.e.,
$$\mathfrak{a}_\lambda:= \{u \in V | \ B_\lambda(u,v)=0 \ \forall v\in V\}$$ 
and let $\mathfrak{b}_\lambda$ be the subspace of $V$ where $B_\lambda$ is non-degenerate.   Consider $\mathfrak{n}_\lambda:=\mathfrak{a}_\lambda \oplus\mathfrak{b}_\lambda\oplus \R X_\lambda$ and $N_\lambda:=\exp(\mathfrak{n}_\lambda)$. We equipped $\mathfrak{a}_\lambda$ with the trivial Lie bracket and $\mathfrak{h}_\lambda:=\mathfrak{b}_\lambda\oplus \R X_\lambda$ with the bracket given by   $$[X,Y]_{\mathfrak{h}_\lambda}:=B_\lambda(X,Y)  X_\lambda.$$ 
We observe that $\mathfrak{h}_\lambda$ is a Heisenberg algebra and we denote by $H_\lambda$ the corresponding Heisenberg group. Let $A_\lambda:=\exp(\mathfrak{a}_\lambda)$.  
The representation $\rho_\lambda$ is trivial on $Z_\lambda:=\exp(\mathfrak{z}_\lambda)$. Thus it factors through $N_\lambda$ and identifying $N_\lambda$ with the product group $A_\lambda \times H_\lambda$ we can write
$$\rho_\lambda(a,n)=\chi(a)\rho_\lambda'(n),$$
where $\chi$ is a unitary character of $A_\lambda$ and $\rho_\lambda'$ is an irreducible unitary representation of $H_\lambda$. Thus, $\rho_\lambda$ cannot be square integrable unless $\mathfrak{a}_\lambda\equiv \{0\}$.
\\

The converse assertion is also true. That is, if  the orbits are maximal (or equivalently, if $B_\lambda$ is non-degenerate on $V$), then $\rho_\lambda$ is square integrable. In fact, $\mathfrak{n}_\lambda=\mathbb{R}X_\lambda\oplus V$. As before, since the character $\chi_\lambda$ is trivial on $\mathfrak{z}_\lambda$, the representation $\rho_\lambda$ acts trivially on $\mathfrak{z}_\lambda$ and defines an irreducible unitary representation in $\widehat{N_\lambda}$. In this case $N_\lambda$ is isomorphic to a Heisenberg group, so $\rho_\lambda$ is a square integrable representation.
\\

This is a particular case of the following general result
(cf. \cite[Theorem 14.2.6]{Wolf}). If $N$ is a connected and simply connected nilpotent Lie group,   the following conditions are equivalent:
\begin{itemize}
\item[(i)] $[\rho_\lambda]\in \widehat{N_\mathrm{sq}}$.
\item[(ii)] The orbit $O(\lambda)$ is completely determined by $\lambda_{|_\mathfrak{z}}$.
\item[(iii)] $B_\lambda$ is non-degenerate on $\mathfrak{n}/\mathfrak{z}$.
\end{itemize}

Finally, we remark that in this context the bilinear form $B_\lambda$ restricted to $V\times V$ defines the symplectic form associated to the Heisenberg group $N_\lambda$, so $\dim(V)=2m$ for some positive integer $m$ and 
the vector space $\mathfrak{m}$ is isomorphic to $\mathfrak{z}\oplus\mathbb{R}^m$.
\\

\bigskip





Now we concentrate on the action of the compact subgroup $K$ of orthogonal automorphisms of $N$ (we make no distinction between automorphisms of $N$ and $\mathfrak{n}$). The group $K$ acts on $\widehat{N}$ in the following way: given $k \in K$ and $(\rho_\lambda, H_{\rho_{\lambda}})  \in \widehat{N}$, $\rho_\lambda^k(x):=\rho_\lambda(k\cdot x)$
defines an irreducible unitary representation of $N$, which may or may not be equivalent to $\rho_\lambda$. On the one hand, let $K_{\rho_{\lambda}}:=\{k\in K | \ \rho_\lambda^k \sim \rho_{\lambda}  \}$ be the stabilizer of $\rho_\lambda$.   
For $k\in K_{\rho_{\lambda}}$, there exists a (unique
up to a unitary factor) unitary operator ${\omega}_\lambda(k)$ on $H_{\rho_{\lambda}}$ which
intertwines $\rho_\lambda$ with $\rho_\lambda^k$, i.e., $\rho_\lambda^k(x)={\omega}_\lambda(k)\rho_\lambda(x){\omega}_\lambda(k)^{-1}$ for all $x\in N$. It defines a genuine unitary
representation of $K_{\rho_{\lambda}}$ on $H_{\rho_{\lambda}}$ (cf. \cite[Lemma 2.3]{BJR1}).
\\

On the other hand, let $K_{X_{\lambda}}$ be the stabilizer of $X_\lambda$ (with respect to the action of $K$ on $\mathfrak{n}$). 
For the case where 
 $\rho_\lambda\in\widehat{N_\mathrm{sq}}$,   it is easy to see that
$K_{\rho_{\lambda}}$ coincides with $K_{X_{\lambda}}$ and we denote them  by $K_\lambda$.
Observe that, for all $u,v\in V$ and for all $k\in K_\lambda$, 
\begin{align*}
B_\lambda(k\cdot u, k\cdot v)&=\langle X_\lambda, [k\cdot u, k\cdot v]\rangle= \langle X_\lambda, k\cdot[u, v]\rangle =\langle k^{-1}\cdot X_\lambda, [u, v]\rangle \\
&=\langle X_\lambda, [u, v]\rangle =B_\lambda( u,  v),
\end{align*}
where in the second equality we used that $K$ acts on $\mathfrak{n}$ by automorphisms. As a result, $K_\lambda$ is a subgroup of the symplectic group $\mathrm{Sp}(V,(B_\lambda)_{|_{V\times V}})$. Moreover, since $K_\lambda$ is compact, we can assume that it is a subgroup of the unitary group $\UU(m)\subset \mathrm{Sp}(V,(B_\lambda)_{|_{V\times V}})$.
At this point we emphasize that the representation ${\omega}_\lambda$ of $K_\lambda$ coincides with the  \textit{metaplectic representation}  associated to the Heisenberg group $N_\lambda$.


\subsection{Localization}\label{subsection loc}

\quad  Let $(\tau,W_\tau)\in\widehat{K}$. We denote by $L^1_{\tau}(N,\mathrm{End}(W_\tau))$ the space of $\mathrm{End}(W_\tau)$-valued integrable functions $F$ on $N$ such that 
\begin{equation*}
F(k\cdot n)=\tau(k)F(n)\tau(k)^{-1} \quad (k\in K, \ n\in N).
\end{equation*}
It is an algebra with the convolution product given by 
\begin{equation*}
(F*G)  (x):=\int_N F(xy^{-1})G(y)dy \quad (x\in N; \  F, G \in L^1_{\tau}(N,\mathrm{End}(W_\tau)) ) .
\end{equation*}
\begin{definition}$(K, N, \tau )$ (or $(K\ltimes N, K,\tau))$ is a \textit{commutative triple} if the algebra $L^1_{\tau}(N,\mathrm{End}(W_\tau))$ is commutative.
\end{definition}

\begin{theorem}\label{heisenberg}\textit{(Reduction to Heisenberg groups)}
Let $N$ be a connected and simply connected real two-step nilpotent Lie group which has a square integrable representation. Let $K$ be a compact subgroup of orthogonal automorphisms of $N$ and let $(\tau,W_\tau)\in\widehat{K}$. Then 
\begin{itemize}
\item[$(i)$] $(K,N)$ is a Gelfand pair if and only if $(K_\lambda,N_\lambda)$ is a Gelfand pair for every  $\rho_\lambda\in \widehat{N_\mathrm{sq}}$ (\textit{scalar case}).
\item[$(ii)$] $(K,N,\tau)$ is a commutative triple if and only if  $(K_\lambda,N_\lambda, \tau_{|_{K_{\lambda}}})$ is a commutative triple for every  $\rho_\lambda\in \widehat{N_\mathrm{sq}}$ (\textit{matrix case}).
\end{itemize}
\end{theorem}
\begin{proof}

\item [$(i)$] We recall that, by hypothesis, $N$ has a Plancherel measure concentrated on $\widehat{N_\mathrm{sq}}$ and thus we apply the argument given in \cite[pages 571-574]{BJR1}.
\\ 
\item [$(ii)$]
A generalization of the ideas given in \cite{BJR1} proves the second statement. In order to do that generalization we use  \cite[Theorem 6.1]{Fulvio} which asserts that $(K,N,\tau)$ is a commutative triple if and only if ${\omega}_\lambda\otimes (\tau_{|_{K_{\rho_{\lambda}}}})$ is multiplicity free for each $\rho_\lambda\in\widehat{N}$.
\\ \\
We assume first that $(K,N,\tau)$ is a commutative triple. In particular, $\omega_\lambda\otimes(\tau_{|_{K_\lambda}})$ is multiplicity free for all  $\rho_\lambda\in \widehat{N_\mathrm{sq}}$. Therefore  \cite[Theorem 6.1]{Fulvio} states that  $(K_\lambda,N_\lambda, \tau_{|_{K_{\lambda}}})$ is a commutative triple for all $\rho_\lambda\in \widehat{N_\mathrm{sq}}$.
\\ \\
Conversely, let $F,G\in L^1_{\tau}(N,\mathrm{End}(W_\tau))$. It is clear that if $\rho\in \widehat{N}$, $\rho(F*G)=\rho(F)\rho(G)$. If we see that $\rho_\lambda(F)$ commutes with $\rho_\lambda(G)$ for all $\rho_\lambda\in \widehat{N_\mathrm{sq}}$, applying the Plancherel theorem (Theorem \ref{Plancherel}) and  \cite[Theorem 14.2.14]{Wolf} we will have shown $F*G=G*F \ a.e$. Let $F\in L^1_{\tau}(N,\mathrm{End}(W_\tau))$.  We know that the operator $\rho_\lambda(F)$ intertwines ${\omega}_\lambda\otimes(\tau_{|_{K_\lambda}})$ with itself (cf. \cite[Lemma 6.2]{Fulvio}). Since, for each $\rho_\lambda\in\widehat{N_\mathrm{sq}}$, $(K_\lambda,N_\lambda, \tau_{|_{K_{\lambda}}})$ is a commutative triple,  
$\omega_\lambda\otimes(\tau_{|_{K_\lambda}})$ is multiplicity free. 
Therefore by Schur lemma, $\rho_\lambda(F)$ is a multiple of the identity operator on each irreducible component of $H_{\rho_\lambda}$ and the conclusion follows immediately.
\end{proof}

With this result we reduce the study of $(K,N,\tau)$ “localizing” the problem to each triple $(K_\lambda,N_\lambda,\tau_{|_{K_{\lambda}}})$, for $\lambda$ in correspondence with $\rho_\lambda\in \widehat{N_\mathrm{sq}}$. 
\\

We will assume that $\rho_\lambda$ is acting on the Fock space $\mathcal{F}_\lambda$ which consist, for $\lambda>0$ (respectively $\lambda<0$), of  holomorphic functions (respectively antiholomorphic) on $\C^m$ square integrable with respect to the measure $e^{-\frac{\lambda}{2}|z|^2}$.
The space $\mathcal{P}(\C^m)$ of polynomials on $\C^m$ is dense on $\mathcal{F}_\lambda$. 
Thereupon we consider the metaplectic action of $\UU(m)$ on $\mathcal{P}(\C^m)$ given by
$(\omega(k)(p))(z):=p(k^{-1}z)$ and
we will work with $\omega_{|_{K_{\lambda}}}$ instead of $\omega_\lambda$.



\subsection{Criteria for a family of two-step nilpotent Lie groups}\label{section torus}


\quad In this subsection we introduce  a subclass of two-step homogeneous nilmanifolds
with an analog construction  to that of H-type groups. This construction is very well explained in the papers  \cite{Lauret} and \cite{Lauret nil} by J. Lauret. Start from a real faithful
representation $(\pi, V)$ of a compact Lie algebra $\mathfrak{g}$ (i.e., $\mathfrak{g} = [\mathfrak{g}, \mathfrak{g}]\oplus \mathfrak{c}$  where $\mathfrak{c}$ is the center of $\mathfrak{g}$ and $\mathfrak{g}^{\prime}:=[\mathfrak{g}, \mathfrak{g}]$
is a compact semisimple Lie algebra).  
Let $\mathfrak{n}:=\mathfrak{g}\oplus V$ be
the two-step nilpotent Lie algebra with center $\mathfrak{g}$ and Lie bracket defined on $V$ by
${\langle[u, v], X \rangle}_\mathfrak{g}:={\langle\pi(X)u, v\rangle}_V$ for all $u, v \in V$, $X \in \mathfrak{g}$, where the inner products $\langle\cdot,\cdot\rangle_\mathfrak{g}$ and $\langle\cdot,\cdot\rangle_V$ are  $\mathrm{ad}(\mathfrak{g})$-invariant and $\pi(\mathfrak{g})$-invariant respectively. These inner products define an inner product $\langle\cdot,\cdot\rangle$ on $\mathfrak{n}$ satisfying 
\begin{gather*}
\langle X,Y\rangle=
\begin{cases}
\langle X,Y\rangle_\mathfrak{g} & \text{if } X,Y\in\mathfrak{g},\\
\langle X,Y\rangle_V & \text{if } X,Y\in V,\\
0 & \text{if } X\in \mathfrak{g} \text{ and } Y\in V.
\end{cases}
\end{gather*}
The simply connected Lie group with Lie algebra $\mathfrak{n}=\mathfrak{g}\oplus V$  is denoted by $N (\mathfrak{g},V)$. This group is endowed with the left-invariant metric
determined by $\langle\cdot,\cdot\rangle$, thus yielding a two-step homogeneous nilmanifold $(N (\mathfrak{g}, V ), \langle\cdot,\cdot\rangle)$. 
\\

Let $G$ be the simply connected Lie group with Lie algebra $\mathfrak{g}^{\prime}$ and let $U$ be the connected component of the identity of the group
of orthogonal intertwining operators of $(\pi,V)$. The group $U$ acts trivially on the center
$\mathfrak{g}$ of $\mathfrak{n}$ and each $g \in G$ acts on $\mathfrak{n}= \mathfrak{g}\oplus V$ by $(\Ad(g), \pi(g))$, where we also denote by $\pi$ the corresponding representation of $G$ on $V$. Let $K$ be the group of orthogonal automorphisms of $N(\mathfrak{g},V)$. If the automorphism group of $\mathfrak{g}$ is exactly the inner automorphism group of $\mathfrak{g}$ ($\mathrm{Aut(\mathfrak{g})}=\mathrm{Inn(\mathfrak{g})}$), then 
$K = G \times U$ (cf. \cite[Theorem 3.12]{Lauret nil}). From now on we will assume $K=G\times U$.
\\ 

For $\lambda\in \mathfrak{n}^*$, let $X_\lambda\in\mathfrak{g}$ be its representative on the center $\mathfrak{g}$. Thus 
\begin{equation*}
\lambda([u,v])=\langle[u,v],X_\lambda\rangle=\langle\pi(X_\lambda)u,v\rangle \quad (u, v \in V).
\end{equation*}
From previous discussion, it holds that the representation $\rho_\lambda$ associated to the coadjoint orbit of $\lambda$ is square integrable if and only if $\Ker({\pi(X_\lambda)})=\{0\}$.
\\ 

Assuming that $N(\mathfrak{g},V)$ has a square integrable representation, note that  $K_\lambda$  is equal to $C_{G}(X_\lambda)\times U$, where $C_{G}(X_\lambda):=\{g\in G | \ \Ad(g)X_\lambda=X_\lambda \}$ is the centralizer of $X_\lambda$ in $G$.
\\ 

 For simplicity, sometimes we set  $N:=N(\mathfrak{g},V)$. Let $\tau\in\widehat{K}$. From Theorem \ref{heisenberg} we know that $(K,N,\tau)$ is a commutative triple if and only if $(\omega\otimes\tau)_{|_{K_\lambda}}$ is multiplicity free for every $\rho_\lambda\in \widehat{N_\mathrm{sq}}$. In the next theorem we state a similar criterion in the case when $\mathfrak{g}$ is a semisimple Lie algebra.

\begin{theorem}\label{torus}
Let  $\mathfrak{g}$  be a semisimple Lie algebra and let $T$ be a maximal torus of $G$.  Assume that $N=N(\mathfrak{g},V)$ has a square integrable representation. The triple $(K,N,\tau)$ is  commutative if and only if $(\omega\otimes\tau)_{|_{T\times U}}$ is multiplicity free. 
\end{theorem}
\begin{proof}
If $(\omega\otimes\tau)_{|_{T\times U}}$ is  multiplicity free, then for each $X_\lambda\in\mathfrak{g}$, $(\omega\otimes\tau)_{|_{K_\lambda}}$ is multiplicity free since $C_G(X_\lambda)$ contains a maximal torus of $G$.
\\ \\
Now we assume that $(K,N,\tau)$ is a commutative triple. We recall that  $X\in\mathfrak{g}$ is said to be a regular element of $\mathfrak{g}$ when $C_{G}(X)$ is a maximal torus  and it is very well known that the set of regular elements of $\mathfrak{g}$ is open and dense in $\mathfrak{g}$.  We only need to find a regular element in $\mathfrak{g}$ corresponding to a square integrable representation of $N$. 
\\ \\
For each $\lambda\in\mathfrak{n}^*$ there is defined a canonical function $P(\lambda)$ which is a homogeneous polynomial on $\mathfrak{n}^*$  called  the \textit{Pfaffian}, which  only depends on $\lambda_{|_{\mathfrak{g}}}$ (for a reference see \cite{Wolf}). It is proved in   \cite[Theorem 14.2.10]{Wolf} that there exists a bijection from $\{\lambda\in \mathfrak{g}^*| \ P(\lambda)\neq 0 \}$ onto $\widehat{N_\mathrm{sq}}$ given by
\begin{equation*}\label{mapa pfaffiano}
\phi(\lambda_{|_\mathfrak{g}})=\rho_\lambda
\end{equation*}
Moreover, Theorem 14.2.14 in \cite{Wolf} asserts that the Plancherel measure of $N$ is concentrated on $\widehat{N_\mathrm{sq}}$ and that its image under the map $\phi^{-1}$ is a positive multiple of $|P(X)|dX$, where $dX$ is the Lebesgue measure in $\mathfrak{g}$.
\\ \\
In particular, the set $\{\lambda\in \mathfrak{g}^*| \ P(\lambda)\neq 0 \}$  is in correspondence  with the set $\{X_\lambda\in \mathfrak{g}| \ \pi(X_\lambda) \text{ is invertible}\}$. Since it is an open set, there is a regular element on it. 
\end{proof}

\begin{corollary}\label{Klambda con regulares}
Let $\mathfrak{g}$ be a non-abelian compact Lie algebra (i.e. $\mathfrak{g}'\not\equiv \{0\}$). Assume that there exists $X\in \mathfrak{g}'$ such that $\pi(X)$ is invertible. The triple $(K,N,\tau)$ is  commutative  if and only if
$(\omega\otimes\tau)_{|_{T\times U}}$ is  multiplicity free, where $T\subset G$ is a maximal torus of $G$.
\end{corollary}
\begin{proof}
As in the previous proof, if $(\omega\otimes\tau)_{|_{T\times U}}$ is  multiplicity free, $(K,N,\tau)$ is  commutative.
Conversely, by the hypothesis  $\{X \in \mathfrak{g}' | \ \pi(X) \text{ is invertible} \}$ is a non-empty open set, so it has a regular element and the conclusion follows as in Theorem \ref{torus}. 
\end{proof}


 The group $N(\mathfrak{g}, V)$ is said \textit{decomposable} if $N(\mathfrak{g}, V)$ is a direct product of Lie groups of the form
 $$N(\mathfrak{g}, V)=N(\mathfrak{h}_1, V_1)\times N(\mathfrak{h}_2, V_2).$$
Otherwise, we will say that $N(\mathfrak{g},V)$ is \textit{decomposable}.
\\ 

 We now recall the classification, due to J. Lauret, of all the Gelfand
pairs of the form $(G\times U, N(\mathfrak{g},V))$ where $N(\mathfrak{g},V)$ is indecomposable and has a square integrable representation  (cf. \cite[Remark 3]{Lauret}). 

\begin{itemize}
\item[(I)] $\left(\SU(2)\times \Sp(n), N(\mathfrak{su}(2),(\mathbb{C}^2)^n)\right), \ n\geq 1$ (Heisenberg-type), where $\mathfrak{su}(2)$ acts on $(\C^2)^n$ as $Im(\HH)$ 
acts component-wise on $\HH^n$  by the quaternion product  on the left side, where $\HH$ denotes the quaternions and $Im(\HH)$ the imaginary quaternions.

\item[(II)] $(\Spin(4)\times \Sp(k_1)\times Sp(k_2), N(\mathfrak{su}(2)\oplus\mathfrak{su}(2), (\mathbb{C}^2)^{k_1}\oplus\R^4 \oplus(\mathbb{C}^2)^{k_2})), \ k_1+k_2\geq 1$, where the real vector space $\R^4=(\C^2\otimes\C^2)_\R$ denotes the standard representation of $\mathfrak{so}(4)=\mathfrak{su}(2)\oplus\mathfrak{su}(2)$ and the first copy of $\mathfrak{su}(2)$ acts only on  $(\mathbb{C}^2)^{k_1}$ and the second one only on  $(\mathbb{C}^2)^{k_2}$.

\item[(III)] $(\Sp(2)\times \Sp(n), N(\mathfrak{sp}(2), (\mathbb{C}^4)^n)), \ n\geq 1$, where $\mathfrak{sp}(2)$ acts component-wise on $(\HH^2)^n$  in the standard way  (identifying $\HH^2$ with $\C^4$).

\item[(IV)] $(\SO(2n), N(\mathfrak{so}(2n),\mathbb{R}^{2n})), \ n\geq 2$ (free two-step nilpotent Lie groups), where $\mathbb{R}^{2n}$ denotes the standard representation of $\mathfrak{so}(2n)$.

\item[(V)] $(\SU(n)\times \Sphere^1, N(\mathfrak{su}(n),\mathbb{C}^n)), \ n\geq 3$, where $\mathbb{C}^n$ denotes the standard representation of $\mathfrak{su}(n)$ regarded as a real representation.

\item[(VI)] $(\SU(n)\times \Sphere^1, N(\mathfrak{u}(n),\mathbb{C}^n)), \ n\geq 3$, where $\mathbb{C}^n$ denotes the standard representation of $\mathfrak{u}(n)$ regarded as a real representation.

\item[(VII)] $(\SU(2)\times \UU(k)\times \Sp(n), N(\mathfrak{u}(2),(\mathbb{C}^2)^k\oplus(\mathbb{C}^2)^n)), \ k\geq 1, n\geq 0$, where the center of $\mathfrak{u}(2)$ acts non-trivially only on $(\mathbb{C}^2)^k$, in fact, $(\mathbb{C}^2)^n$ denotes the representation of $\mathfrak{su}(2)$ described in  item (I) and $\mathfrak{u}(2)$ acts component-wise on $(\mathbb{C}^2)^k$ in the standard way regarded as a real representation.

\item[(VIII)] $(G\times U,N(\mathfrak{g},V))$ where:
\begin{itemize}
\item $\mathfrak{g}:=\mathfrak{su}(m_1)\oplus...\oplus\mathfrak{su}(m_\beta)\oplus\mathfrak{su}(2)\oplus...\oplus\mathfrak{su}(2)\oplus \mathfrak{c}$, with $\alpha$  copies of $\mathfrak{su}(2)$, $m_i\geq 3$ for all $1\leq i\leq \beta$ and $\mathfrak{c}$ is an abelian component.
\item $V:=\C^{m_1}\oplus...\oplus \C^{m_\beta}\oplus\C^{2k_1+2n_1}\oplus...\oplus\C^{2k_\alpha+2n_\alpha}$, where $k_j\geq 1$ and $n_j\geq 0$ for all $1\leq j\leq \alpha$.
\item $\mathfrak{g}$ acts on $V$ as follows: for each $1\leq i\leq \beta+\alpha$, $\mathfrak{c}$ has a maximal subspace, denoted by $\mathfrak{c}_i$ and with $\dim(\mathfrak{c}_i)=1$,  acting non-trivially only on one component of $V$. For $1\leq i\leq \beta$,
 $\mathfrak{su}(m_i)\oplus \mathfrak{c}_i$  acts non-trivially only on  $\C^{m_i}$ (as the representation stated in item (VI)) and for $\beta+1<i\leq \beta+\alpha $, $\mathfrak{su}(2)\oplus\mathfrak{c}_i$  acts non-trivially only on $\C^{2k_i+2n_i}$ (as the representation stated in item (VII)).
\item $U:=\Sphere^1\times...\times \Sphere^1\times \UU(k_1)\times \Sp(n_1)\times...\times \UU(k_\alpha)\times \Sp(n_\alpha)$, with $\beta$ copies of $\Sphere^1$.
\end{itemize}

\item[(IX)] $(\UU(n), N(\R,\mathbb{C}^n)), \ n\geq 1$ (Heisenberg group).
\end{itemize}

\begin{remark}
We observe that for all the cases stated in the above list there exists $X\in\mathfrak{g}'$ satisfying $\Ker(\pi(X))=\{0\}.$
\end{remark}

Our goal is to consider, for each case, an arbitrary irreducible unitary representation $\tau$ of $K=G\times U$ and determine if the resulting triple is commutative. The last item of the list
was studied in \cite{Yakimova} where it was proved that the triples $(\UU(n),N(\R,\mathbb{C}^n),\tau)$ are commutative for all $\tau\in\widehat{\UU(n)}$, i.e., when $N$ is the Heisenberg group and $K$ is the unitary group we have a strong Gelfand pair.
\\ 

We remark that we are excluding two of the cases given in  \cite[Remark 3]{Lauret} since the corresponding nilpotent group does not have square integrable representations. They still may give rise to commutative triples.


\subsection{Background about tensor product of irreducible representations relative to the symplectic group} \label{subsec tensor product of rep}

\quad  A \textit{partition} is a monotone decreasing finite sequence $\sigma=(\sigma_1,\sigma_2,...)$ where its entries $\sigma_i$ are nonnegative integers. The set of all partitions is denoted by $\mathbb{P}$. The \textit{length} of $\sigma$, denoted by $l(\sigma)$, is the number of nonzero entries of $\sigma$, and the \textit{size} of $\sigma$, denoted by $|\sigma|$, is the sum of its entries. A partition $\sigma$ is often identified with its \textit{Young diagram}, which is a left-justified array of $|\sigma|$ cells with $\sigma_i$ cells in the $i$-th row. The conjugate partition $\sigma^t$
of $\sigma$ is the partition whose diagram is obtained 
by reflecting the diagram of $\sigma$ along the main diagonal. For two partitions $\sigma$ and $\varsigma$, we write $\varsigma\subset\sigma$ if $\varsigma_i\leq\sigma_i$ for all $i$ and the skew diagram $\sigma /\varsigma $ is defined to be the set  that remains when we make the difference of the two diagrams. The size of this skew diagram is defined by $|\sigma / \varsigma|=|\sigma|-|\varsigma|$. The skew diagram $\sigma / \varsigma$ is said to be a \textit{horizontal $k$-strip} if it contains at most one cell in each column and $|\sigma / \varsigma |=k$. Note that $\sigma / \varsigma$ is a horizontal strip if and only if $\sigma_1\geq\varsigma_1\geq\sigma_2\geq\varsigma_2\geq...$ .
\\

Let $\mathfrak{h}$ be the Cartan subalgebra of $\mathfrak{sp}(n)$ consisting of the diagonal matrices (over $\C$) of the form
\begin{equation*}
\mathfrak{h}=\left\lbrace H=\begin{pmatrix}  
  h_1&    \\
     & \ddots &    \\
     &        & h_n&  \\
     &        &    & -h_1 & \\
     &        &    &      & \ddots &\\
     &        &    &      &        &  -h_n
\end{pmatrix} | \ h_1, ..., h_n \in \C \right\rbrace .
\end{equation*}
Let $\delta_{i,j}$ denote the Kronecker delta. Writing 
$H_i:=\delta_{i,i}-\delta_{n+i,n+i}$, we have that $\mathfrak{h}$ is a complex vector space generated by $\{ H_1,...,H_n \}$. Let 
$\{L_1,...,L_n\}$
be its dual basis in the dual space $\mathfrak{h}^*$, so $\langle L_i,H\rangle=h_i$ for all $H\in\mathfrak{h}$. The fundamental weights of $\mathfrak{sp}(n)$ are given by $L_1, L_1+L_2,..., L_1+...+L_n$ and the simple roots are $L_1-L_2, ,..., L_{n-1}-L_n, 2L_n$. From the theorem of the highest weight,  every  irreducible representation of $\mathfrak{sp}(n)$ is in correspondence with a nonnegative integer linear combination of the fundamental weights. Hence $\eta \in \widehat{\Sp(n)}$ can be parametrized in terms of the weights $\{L_i\}$ as $(\eta_1,...,\eta_n)$  where  $\eta_i\in\Z_{\geq 0}$ $\forall i$ and $\eta_1\geq \eta_2\geq ... \geq \eta_n$. (For a reference see, for example, \cite{Knapp} or \cite{Fulton y Harris}.)
\\ 

By abuse of notation, for each $\eta\in\widehat{\Sp(n)}$ we will denote its partition with the same letter $\eta$. 
Apart from that, we will denote the representation $\eta\in\widehat{\Sp(n)}$ by $\eta_{(\eta_1,...,\eta_n)}$ to emphasize that the representation $\eta$ is in correspondence with the partition $(\eta_1,...,\eta_n)$.  
\\ 

The decomposition of the tensor product of irreducible representations of $\mathfrak{sp}(n)$ is developed in \cite{Koike} and in the recent works \cite{Howe} and \cite{Okada}.
\\ 

The irreducible unitary representations of $\Sp(n)$ on the space of the homogeneous polynomials over $\C^{2n}$ of degree $r$ and $s$ respectively are associated to the partitions $(r)$ and $(s)$ of length one and they are denoted by $\eta_{(r)}$ and $\eta_{(s)}$. The formula of the decomposition of their tensor product is given in \cite[page 510 ]{Koike} for $r\geq s$:
\begin{equation}\label{tensor de simetricas en sp}
\eta_{(r)}\otimes\eta_{(s)}=\bigoplus_{j=0}^s\bigoplus_{i=0}^j\eta_{(r+s-j-i,j-i)}.
\end{equation}

The tensor product of a 
representation $\eta_{(1,...,1)}$ of length $r>1$ (denoted also by $\eta_{(1^r)}$) by a representation $\eta_{(s)}$ (with $s>1$) is  given in \cite[page 510]{Koike}  as well:
\begin{equation}\label{tensor de fundamental y simetrica}
\eta_{(1^r)}\otimes \eta_{(s)}=
\eta_{(s+1,1^{r-1})}\oplus\eta_{(s,1^r)}\oplus
\eta_{(s-1,1^{r-1})}\oplus\eta_{(s,1^{r-2})}.
\end{equation}


Finally, we enunciate a “universal” Pieri rule due to S. Okada.
\begin{theorem}\cite[Proposition 3.1]{Okada}
For an arbitrary irreducible representation (or partition) $\eta$ and a nonnegative integer $s$,
\begin{equation}\label{universal Pieri law}
\eta\otimes \eta_{(s)} \ =\sum_{\sigma\in\mathbb{P} \ \text{of length at most } n}M_{\eta,s}^\sigma \ \eta_{(\sigma)},
\end{equation}
where $M_{\eta,s}^\sigma$ denotes the number of partitions $\varsigma$ such that $\eta / \varsigma$ and $\sigma / \varsigma$  are horizontal strips and $|\eta / \varsigma|+|\sigma / \varsigma|=s$.
\end{theorem}

\begin{corollary}\label{coro 1 sp} Let $\eta$ be
an arbitrary irreducible representation of $\Sp(n)$, then
$\eta$ appears in the decomposition into irreducible factors of $\eta\otimes\eta_{(2)}$.
\end{corollary}
\begin{proof}
We use the “universal” Pieri rule with $s=2$.  The cardinal number
$M_{\eta,2}^\eta$ is not zero since considering  $\varsigma=(\eta_1,...,\eta_{m}-1)$ we have that $\eta / \varsigma$ is a horizontal strip (there is only one box in the Young diagram) and $2|\eta / \varsigma|=2$.
\end{proof}

\begin{corollary}\label{coro 2 sp}
Let $\eta\in\widehat{\Sp(n)}$ corresponding to the partition $(\eta_1,...,\eta_m)$, with $\eta_m\neq 0$, $0\leq m\leq n$. The representation $\eta\otimes\eta_{(s)}$ is multiplicity free for all $s\in\Z_{\geq 0}$ if and only if $\eta_i=\eta_j$ for all $1\leq i,j\leq m$.
\end{corollary}
Note that, from (\ref{tensor de simetricas en sp}) and (\ref{tensor de fundamental y simetrica}), the corollary holds in the particular cases $\eta=\eta_{(r)}$ (for some $r\in\Z_{\geq 0}$) and $\eta=\eta_{(1^{r'})}$ (for some $r'\in\Z_{\geq 0}$).
\begin{proof}
We assume first that $s=2$ and that $\eta_i>\eta_{i+1}$, i.e. $\eta_i -1\geq\eta_{i+1}$, for some $1\leq i\leq m$. From (\ref{universal Pieri law}), $\eta$ appears in the decomposition of $\eta\otimes \eta_{(2)}$ at least twice: considering $\varsigma_1=(\eta_1,...,\eta_m-1)$ 
and $\varsigma_2=(\eta_{1},...,\eta_{i-1},\eta_i-1,\eta_{i+1},...,\eta_k)$ we have that $\eta /\varsigma_1$ and $\eta /\varsigma_2$ are 1-horizontal strips. 
\\ \\
Therefore we have proved that $\eta$ must satisfy necessarily the condition $\eta_i=\eta_j$ for all $1\leq i,j\leq m$. Now we will prove that this condition is sufficient.
\\ \\
Let  $(a,a,...,a)$ be the partition of length $0\leq m\leq n$ associated to $\eta$ for some $a\in\N$. Let $\varsigma$ be a partition such that $\eta / \varsigma$ is a horizontal strip. Since $\varsigma$ must satisfy $\eta_1\geq\varsigma_1\geq...\geq\varsigma_{m-1}\geq\eta_m\geq\varsigma_m$, we have that  $\varsigma$ is equal to $\varsigma_j:=(a,...,a,a-j)$ for some $0\leq j\leq a$. Now we fix an arbitrary nonnegative integer $s$ and let $\sigma$ be an irreducible representation which appears in the decomposition of $\eta\otimes\eta_{(s)}$. Accordingly, there is $0 \leq j\leq a$ such that $\sigma /\varsigma_j$ is an $(s-j)$-horizontal strip. Note that the partition $(\sigma_1,...,\sigma_l)$ associated to $\sigma$ must satisfy: $l\leq m+1$, $\sigma_2=...=\sigma_{m-1}=a$,  $\sigma_1\geq a$, $(a-j)\leq \sigma_{m}\leq a$ and $0\leq\sigma_{m+1}\leq (a-j)$. Therefore $|\sigma / \varsigma_j|=\sigma_1-a+\sigma_{m}-(a-j)+\sigma_{m+1}$ and thus $\sigma_1+\sigma_{m}+\sigma_{m+1}=s+2a-2j$. From the last equality, $\sigma$ appears only once in the decomposition of $\eta\otimes\eta_{(s)}$: indeed, if we assume that it appears two times (or more), then there are $\varsigma_j$ and $\varsigma_k$, as above, with $j\neq k$, such that $\sigma /\varsigma_j$ and $\sigma /\varsigma_k$ are $(s-j)$ and $(s-k)$ horizontal strips respectively, hence $s+2a-2j=s+2a-2k$ (so $j=k$) and we have a contradiction.
\end{proof}

\section{Case by case analysis}\label{section case by case}

 In this paragraph we explore each of the cases given in Section \ref{section torus} in order to derive commutative triples.






\subsection*{Case (I): Heisenberg type}


\quad Let $\HH$ denote the quaternions and $Im(\HH)$ the imaginary quaternions. For $n\geq 1$, the group $N(\mathfrak{su}(2),(\C^2)^n)$ is the H-type group with Lie algebra $Im(\HH)\oplus\HH^n$ where the real representation $\pi$ of $Im(\HH)$ on $\HH^n$ is given by the quaternion product on the left side on each component,   
\begin{equation}
\pi(z)(v):=(zv_1,...,zv_n) \quad (v=(v_1,...,v_n)\in\HH, \  z\in Im (\HH))
\end{equation}
i.e. $\mathfrak{su}(2)$ acts on $(\C^2)^n$ as $Im(\HH)$   acts naturally on each coordinate of $\HH^n$ (identifying $Im(\HH)$ with $\mathfrak{su}(2)$ and $\HH^n$ with $(\C^2)^n$).
\\

The unitary group of intertwining operators of $\pi$ is $\Sp(n)$ that acts on $\HH^n$ on the right side and  $K=\SU(2)\times \Sp(n)$. 
A maximal torus on $\SU(2)$ is given by $\Toro^1:=\{\left(\begin{smallmatrix} e^{i\theta} & 0\\ 0 & e^{-i\theta} \end{smallmatrix}\right) | \ \theta\in\R \}$. Note that the action of $\Toro^1\times \Sp(n)$ on  $\mathcal{P}((\C^2)^{n})$ (the space of polynomials on $(\C^2)^n$) is $\C$-linear whereas the action of $K$ is not.
\\ 

The metaplectic representation $\omega$ of $\Toro^{1}\times \Sp(n)$ on  $\mathcal{P}((\C^2)^n)$ decomposes into
\begin{equation}
\omega=\bigoplus_{s \in\Z_{\geq 0} }\chi_{s}\otimes\eta_{(s)}, 
\end{equation}
where $\chi_{s}(\theta):=e^{-is\theta}$ and $\eta_{(s)}$ denotes the irreducible representation of $\Sp(n)$ on the space $\mathcal{P}_s((\C^2)^n)$ of homogeneous polynomials of degree $s$.
\\ 

A representation $\tau\in\widehat{\SU(2)\times \Sp(n)}$ is given by the tensor product $\tau=\nu_k\otimes \eta$, where $(\nu_k,\mathcal{P}_k(\C^{2}))\in \widehat{\SU(2)}$ is the well known representation of $\SU(2)$ on the homogeneous polynomials on $\C^{2}$ of degree $k$ and $\eta\in \widehat{\Sp(n)}$. When we restrict $\tau$ to $\Toro^1\times \Sp(n)$ we have
\begin{equation}
(\nu_k\otimes \eta)_{|_{\Toro^1\times \Sp(n)}}=\left( \bigoplus_{i=0}^k\chi_{k-2i} \right)\otimes \eta.
\end{equation}
In consequence
\begin{equation}\label{wxt I}
(\omega\otimes\nu_k\otimes \eta)_{|_{\Toro^1\times \Sp(n)}}=\bigoplus_{s=0}^{\infty}\left[\left( \oplus_{i=0}^k\chi_{s+k-2i} \right)\otimes \eta_{(s)} \right] \otimes\eta.
\end{equation}

\begin{proposition}\label{teo ej i}
The triple $(\SU(2)\times \Sp(n), N(\mathfrak{su}(2),(\C^2)^n),\tau)$ is commutative if and only if:  
\begin{itemize}
\item[$(i)$]  $\tau\in\widehat{\SU(2)}$ or
\item[$(ii)$]  $\tau\in \widehat{\Sp(n)}$
and corresponds to a partition of the form $(a,a,...,a)$, of length at most $n$. 
\end{itemize}
\end{proposition}

\begin{proof}
Let $\tau=\nu_k\otimes \eta\in\widehat{\SU(2)\times \Sp(n)}$.
\\  \\
From (\ref{wxt I}), if $\tau=\nu_k$, 
the assertion is easily proved since 
\begin{equation*}
(\omega\otimes\nu_k)_{|_{\Toro^1\times \Sp(n)}}=\bigoplus_{s=0}^{\infty}\left( \oplus_{i=0}^k\chi_{s+k-2i} \right)\otimes \eta_{(s)} 
\end{equation*}
is multiplicity free. 
For the case $\tau=\eta$, from (\ref{wxt I}) we arrive at 
\begin{equation*}
(\omega\otimes \eta)_{|_{\Toro^1\times \Sp(n)}}=\bigoplus_{s=0}^{\infty} \chi_{s} \otimes \eta_{(s)}  \otimes\eta.
\end{equation*}
It is multiplicity free if and only if $\eta_{(s)}  \otimes\eta$ is multiplicity free for all $s\in\Z_{\geq 0}$. From Corollary \ref{coro 2 sp} this holds if and only if $\eta$ corresponds to a partition of the form $(a,a,...,a)$ of length at most $n$.  
\\ \\
Now we assume  
$\eta$ and $\nu_k$ are both nontrivial and from Theorem \ref{torus}, we will prove that $(\omega\otimes\nu_k\otimes \eta)_{|_{\Toro^1\times \Sp(n)}}$ is not multiplicity free. 
\\ \\
If we consider $s=0$ in (\ref{wxt I}), we pick  $\chi_k\otimes\eta$. On the other hand, if we consider $s=2$, we pick $\chi_k\otimes\eta_{(2)}\otimes\eta$. 
From Corollary \ref{coro 1 sp}, $\eta$ appears in the decomposition into irreducible factors of the tensor product $\eta_{(2)}\otimes \eta$. Therefore $\eta$ appears at least twice in the decomposition of $(\omega\otimes\nu_k\otimes \eta)_{|_{\Toro^1\times \Sp(n)}}$.
\end{proof}

\subsection*{Case (II)}

\quad In this case $\mathfrak{g}=\mathfrak{su}(2)\oplus\mathfrak{su}(2)$ and it acts on $V=(\C^2)^{k_1}\oplus\R^4\oplus(\C^2)^{k_2}$ (with $k_1+k_2 \geq 1$) by
\begin{itemize}
\item[$(i)$] the first copy $\mathfrak{su}(2)$ of $\mathfrak{g}$ acts only on $(\C^2)^{k_1}$ as $\mathfrak{sp}(1)$ acts on $\HH^{k_1}$  in the natural way, component-wise, whereas the second copy $\mathfrak{su}(2)$ of $\mathfrak{g}$ acts analogously only on $(\C^2)^{k_2}$ and
\item[$(ii)$] $\mathfrak{g}=\mathfrak{so}(4)$ acts naturally on $\R^4$ (since $\mathfrak{su}(2)\oplus\mathfrak{su}(2)= \mathfrak{so}(4)$).
\end{itemize}

Thus, the connected unitary group of intertwining operators of this action is $U=\Sp(k_1)\times \Sp(k_2)$.
\\ 

Let $\Toro^1$ be a torus of $\SU(2)$. The metaplectic representation of $\Toro^1\times \Toro^1\times U$ on $\mathcal{P}((\mathbb{C}^2)^{k_1}\oplus\R^4\oplus(\mathbb{C}^2)^{k_2})$ decomposes in the following way:
\begin{itemize}
\item[$(i)$] The action of $\Toro^1\times \Sp(k_1)$ on  $\mathcal{P}((\C^{2})^{k_1})$ decomposes without multiplicity into irreducible representations  as
\begin{equation} \label{meta iii a}
\bigoplus_{r\in\Z_{\geq 0}}\chi_r\otimes\eta_{(r)} ,
\end{equation}
where $\chi_r(\theta)=e^{-ri\theta}$ (for all $\theta\in \Toro^1$) and $\eta_{(r)}$ is the classical representation of $\Sp(k_1)$ on $\mathcal{P}_r((\C^2)^{k_1})$.

The action of $\Toro^1\times \Sp(k_2)$ on  $\mathcal{P}((\C^{2})^{k_2})$ decomposes analogously
\begin{equation}\label{meta iii b}
\bigoplus_{s\in\Z_{\geq 0}}\chi_{s}\otimes\eta_{(s)},
\end{equation}
where here $\eta_{(s)}$ is the classical representation of $\Sp(k_2)$ on $\mathcal{P}_s((\C^2)^{k_2})$.
\item[$(ii)$] The action of $\Toro^1\times \Toro^1=\{\left(\begin{smallmatrix}
e^{i\theta_1} & 0\\
0& e^{i\theta_2}
\end{smallmatrix}\right)| \ \theta_1, \theta_2\in \R \}$ on $\mathcal{P}(\C^2)$ decomposes without multiplicity into the following sum of characters,
\begin{equation}\label{meta iii c}
\bigoplus_{l_1, l_2\in\Z_{\geq 0}}\chi_{(l_1, l_2)}, 
\end{equation}
where  $\chi_{(l_1, l_2)}(\theta_1,\theta_2)=e^{-l_1i\theta_1}e^{-l_2i\theta_1}$ (for all $(\theta_1,\theta_2)\in \Toro^1\times \Toro^1$).
\end{itemize}

Therefore from (\ref{meta iii a}), (\ref{meta iii b}) and (\ref{meta iii c}),
\begin{equation}\label{meta iii}
\omega_{|_{\Toro^1\times \Toro^1\times U}}=\bigoplus_{r,s,l_1, l_2\in\Z_{\geq 0}}\chi_{(r+l_1, s+l_2)}\otimes\eta_{(r)}\otimes\eta_{(s)}.
\end{equation}

Let $\tau=\nu_m\otimes\nu_n\otimes\eta_1\otimes\eta_2$ be an irreducible unitary representation of $K=\SU(2)\times \SU(2)\times U$, where $\nu_m$ and $\nu_n$ are the classical irreducible unitary representations of $\SU(2)$ on $\mathcal{P}_m(\C^2)$ and $\mathcal{P}_n(\C^2)$ respectively, and  $\eta_1$ and $\eta_2$ are arbitrary irreducible unitary representations of $\Sp(k_1)$ and $\Sp(k_2)$ respectively. When we restrict $\tau$ to $\Toro^1\times \Toro^1\times U$ we obtain the following decomposition,
\begin{equation}\label{tau iii}
\left(\bigoplus_{i=0}^m\chi_{m-2i}\right)\otimes
\left(\bigoplus_{j=0}^n\chi_{n-2j}\right)\otimes
\eta_1\otimes\eta_2. 
\end{equation}

Therefore from (\ref{meta iii}) and (\ref{tau iii}),
\begin{equation}\label{omega tau iii}
(\omega\otimes\tau)_{|_{\Toro^1\times \Toro^1\times U}}=\bigoplus_{r,s,l_1, l_2\in\Z_{\geq 0}}\left(\oplus_{i=0}^m\oplus_{j=0}^n\chi_{(r+l_1+m-2i, s+l_2+n-2j)}\right)\otimes\left(\eta_{(r)}\otimes\eta_1\right)\otimes\left(\eta_{(s)}\otimes\eta_2\right).
\end{equation}

\begin{proposition}
The triple $(\SU(2)\times \SU(2)\times \Sp(k_1)\times \Sp(k_2), N(\mathfrak{su}(2)\oplus\mathfrak{su}(2),(\C^2)^{k_1}\oplus\R^4\oplus(\C^2)^{k_2}),\tau)$ is commutative if and only if $\tau$ is the trivial representation.
\end{proposition}
\begin{proof}
Let $\tau=\nu_m\otimes\nu_n\otimes\eta_1\otimes\eta_2$ as above. If $\nu_m$ is nontrivial, i.e. $m\geq 1$, considering in (\ref{omega tau iii}) $l_1=0$, $i=0$ and then $l_1=2$, $i=1$, we obtain multiplicity in the decomposition into irreducible representations of $(\omega\otimes\tau)_{|_{\Toro^1\times \Toro^1\times U}}$. Analogously, if $\nu_n$ is nontrivial, i.e. $n\geq 1$, we have multiplicity in (\ref{omega tau iii}).
\\ \\
If $\eta_1\in\widehat{\Sp(k_1)}$ is nontrivial, by the same argument given in Proposition \ref{teo ej i}  (that uses Corollary \ref{coro 1 sp}), considering on the one hand $r=0$, $l_1=2$ and on the other hand $r=2$, $l_1=0$ on (\ref{omega tau iii}), we have multiplicity in $(\omega\otimes\tau)_{|_{\Toro^1\times \Toro^1\times U}}$. The same holds for a nontrivial $\eta_2\in\widehat{\Sp(k_2)}$.
\end{proof}


\subsection*{Case (III)}

\quad  This case consists of $\mathfrak{g}=\mathfrak{sp}(2)$ and $\mathfrak{n}=\mathfrak{sp}(2)\oplus (\HH^2)^n$ (with $n\geq 1$), where we are identifying $\C^4$ with $\HH^2$ (as real vector spaces). The real representation $\pi$ of $\mathfrak{sp}(2)$  on $(\HH^2)^n$ is similar  to the one described in the case (I):
\begin{equation}\label{pi ej iv}
\pi(z)(v):=(zv_1,...,zv_n) \quad (v=(v_1,...,v_n)\in \HH^2,  \ z\in \mathfrak{sp}(2)).
\end{equation}

As the group of unitary intertwining operators of $(\pi,(\HH^2)^n)$ is $\Sp(n)$ (acting on the right side as a $\C$-linear action), the group $K$ is $\Sp(2)\times \Sp(n)$. 
\\ 

Let $\Toro^2:=\{ \left( \begin{smallmatrix}
e^{i\theta_1} &  0  \\
0 & e^{i\theta_2}  
\end{smallmatrix} \right)  | \ \theta_1,\theta_2\in\R \}$ be a maximal torus on $\Sp(2)$. 
\\ 

Since the natural action of $\mathfrak{sp}(2)$ on $\HH^2$ is irreducible, Schur's lemma implies that  each intertwining operator $A\in \Sp(n)$ of the action of $\mathfrak{sp}(2)$ on $\mathcal{P}((\HH^{2})^n)$, $A:(\HH^2)^n\rightarrow (\HH^2)^n$, has the following  matrix representation
\begin{equation*}
[A]=\begin{pmatrix} 
a_{11}I & a_{12}I & ... & a_{1n}I\\ 
a_{21}I & a_{22}I & ... & a_{2n}I\\
... & ... & ... & ...\\
... & ... & ... & ...\\
a_{n1}I & a_{n2}I & ... & a_{nn}I
\end{pmatrix}, \textrm{ where } I=\begin{pmatrix} 
1 & 0\\ 
0 & 1
\end{pmatrix}, \text{ and } a_{i,j}\in \HH.
\end{equation*}
We can deduce that the action of $\Sp(n)$ on $\mathcal{P}((\C^4)^n)$  splits as $\mathcal{P}((\C^2)^n)\otimes \mathcal{P}((\C^2)^n)$ and also it can be written as $\bigoplus_{r,s\in \Z_{\geq 0}}\mathcal{P}_r((\C^2)^n)\otimes \mathcal{P}_s((\C^2)^n)$. 
\\ 

Therefore the metaplectic representation $\omega$ of  $\Toro^2\times \Sp(n)$ on $\mathcal{P}((\C^4)^{n})$ decomposes into
\begin{equation*}
\omega_{|_{\Toro^2\times \Sp(n)}}=
\bigoplus_{r,s\in\Z_{\geq 0}}\chi_{(r,s)}\otimes
\eta_{(r)}\otimes\eta_{(s)},
\end{equation*}
where $\chi_{(r,s)}(\theta_1,\theta_2)=e^{-ir\theta_1}e^{-is\theta_2}$ and where $\eta_{(r)}$ and $\eta_{(s)}$ denote the classical representations of $\Sp(n)$ on the homogeneous polynomials on $(\C^2)^n$  of degree $r$ and $s$ respectively. Moreover, from (\ref{tensor de simetricas en sp}), we obtain the multiplicity free decomposition
\begin{equation}\label{meta ej iv}
\omega=\bigoplus_{r,s\in\Z_{\geq 0}}\chi_{(r,s)}\otimes\left(\bigoplus_{j=0}^s\bigoplus_{i=0}^j\eta_{(r+s-j-i,j-i)}\right).
\end{equation}

\begin{proposition}
The triple $(\Sp(2)\times \Sp(n), N(\mathfrak{sp}(2),(\HH^2)^n), \tau)$  is  commutative if and only if $\tau$ is the trivial representation. 
\end{proposition}
\begin{proof}
Let $\tau\in\widehat{\Sp(2)\times \Sp(n)}$. Then $\tau=\kappa\otimes \eta$, where $\kappa \in \widehat{\Sp(2)}$ and $\eta\in \widehat{\Sp(n)}$.
\\ \\
First, let $\eta\in\widehat{\Sp(n)}$ be nontrivial. Fix $r=1$ and $s=1$ in (\ref{meta ej iv}). On the one hand take in (\ref{meta ej iv}) $i=j=0$ and on the other hand, $i=j=1$ . As a result,  $\chi_{(1,1)}\otimes(\kappa_{|_{\Toro^2}})\otimes\eta$ and $\chi_{(1,1)}\otimes(\kappa_{|_{\Toro^2}})\otimes(\eta_{(2)}\otimes\eta)$ appear in $(\omega\otimes\tau)_{|_{\Toro^2\times \Sp(n)}}$. From Corollary \ref{coro 1 sp} we know that $\eta$ appears in the decomposition of $\eta_{(2)}\otimes\eta$. As we expected, $(\omega\otimes\tau)_{|_{\Toro^2\times \Sp(n)}}$ is not multiplicity free. This leads us to assume the component $\eta$ to be trivial.
\\ \\
Now, let $\kappa\in\widehat{\Sp(2)}$ be nontrivial. 
We can associate to $\kappa$  the partition $(\kappa_1,\kappa_2)$. 
From the theorem of the highest weight, every weight of $\kappa$ is given by the highest weight minus nonnegative sums of the simple roots $L_1-L_2$ and $2L_2$. 
Let $\chi_{(\kappa_1,\kappa_2)}$ and $\chi_{(\kappa_1-l, \kappa_2+l-2m)}$ be two different characters in the decomposition of $\kappa_{|_{\Toro^2}}$ for some nonnegative integers $l$ and $m$. In the decomposition of $\omega_{|_{\Toro^2\times \Sp(n)}}\otimes\kappa_{|_{\Toro^2}}$ appear (in particular) the following sums of irreducible representations
\begin{equation*}
\left(\bigoplus_{r,s \in \Z_{\geq 0}}\chi_{(\kappa_1+r,\kappa_2+s)}\otimes\eta_{(r)}\otimes\eta_{(s)}\right)\oplus \left(\bigoplus_{r',s' \in \Z_{\geq 0}}\chi_{(\kappa_1-l+r',\kappa_2+l-2m+s')}\otimes\eta_{(r')}\otimes\eta_{(s')}\right).
\end{equation*}
We choose $r'=l$, $s'=2m$, $r=0$ and $s=l$.  They satisfy 
\begin{equation}\label{eleccionparam}
     \left\{
	       \begin{array}{ll}
		 r=-l+r' \\
		 s=l-2m+s'
	       \end{array}
	     \right.
\end{equation}
The representation $\eta_{(l)}$  appears in $\eta_{(r)}\otimes\eta_{(s)}$ and in $\eta_{(r')}\otimes\eta_{(s')}$. This holds since from (\ref{tensor de simetricas en sp}),
\begin{align*}
\eta_{(r)}\otimes\eta_{(s)}&=\eta_{(0)}\otimes\eta_{(l)}=\eta_{(l)},\\
\eta_{(r')}\otimes\eta_{(s')}&=\eta_{(l)}\otimes\eta_{(2m)}=\oplus_{j=0}^{max\{l,2m\}}\oplus_{i=0}^j\eta_{(l+2m-j-i,j-i)},
\end{align*}
and taking $j=i=m$ we see that $\eta_{(l)}$ appears in $\eta_{(r')}\otimes\eta_{(s')}$.
\\ \\
In conclusion, when $\tau$ is not the trivial representation, $(\omega\otimes\tau)_{|_{\Toro^2\times \Sp(n)}}$ is not multiplicity free and by Theorem \ref{torus},  the triple is not commutative.
\end{proof}


\subsection*{Case (IV): free two-step nilpotent Lie groups}

\quad We will study the pairs $(\SO(2n), N(\mathfrak{so}(2n),\mathbb{R}^{2n}))$, for  $n\geq 2$. Here the representation $\pi$ is the natural action of $\mathfrak{so}(2n)$ on $\mathbb{R}^{2n}$ and therefore 
$K=\SO(2n)$.
\\ 

We fix a maximal torus on $\SO(2n)$,
\begin{equation*}
\Toro^{n}:=\left\lbrace \begin{pmatrix}  
  \cos (\theta_1)& -\sin(\theta_1)  & &0 & 0&\\
  \sin(\theta_1) & \cos(\theta_1) &   &0& 0& \\
                                & & \ddots&  &  & \\
           0& 0&   &\cos (\theta_n)& -\sin(\theta_n)\\
    0 & 0 &   & \sin(\theta_n) & \cos(\theta_n)
\end{pmatrix} | \ \theta_1, ..., \theta_n \in \R \right\rbrace .
\end{equation*}

The metaplectic representation of $\Toro^n$ on $\mathcal{P}(\C^n)$ is the usual one and it decomposes without multiplicity into a direct sum of characters
\begin{equation}
\omega_{|_{\Toro^n}}=\bigoplus_{(k_1,...,k_n)\in \mathbb{Z}_{\geq 0}}\chi_{(k_1,...,k_n)},
\end{equation}
where $(\chi_{(k_1,...,k_n)}(\theta_1,...,\theta_n))(z_1^{k_1}...z_n^{k_n})=e^{-ik_1\theta_1}z_1^{k_1}...e^{-ik_n\theta_n}z_n^{k_n}$.
\\ 

Let $(\tau,W_\tau)\in\widehat{\SO(2n)}$ be a nontrivial representation.
We must analyze the decomposition into irreducible factors of $\tau$ restricted to the torus $\Toro^n$ and then
verify whether $\omega_{|_{\Toro^n}}\otimes\tau_{|_{\Toro^n}}$ is multiplicity free or not:
\begin{itemize}
\item[(i)] If $\tau_{|_{\Toro^n}}$ is not multiplicity free, then $\omega_{|_{\Toro^n}}\otimes\tau_{|_{\Toro^n}}$ is not multiplicity free either.
\item[(ii)] Assume that $\tau_{|_{\Toro^n}}$ is multiplicity free. Since $\tau$ is not the trivial representation, $\tau_{|_{\Toro^n}}$ decomposes into a sum of characters (as many as its dimension).
Let $\chi_{(r_1,...,r_n)}$ and $\chi_{(s_1,...,s_n)}$ be two different characters on $\tau_{|_{\Toro^n}}$ determined by the $n$-tuples of integers $(r_1,...,r_n)$ and $(s_1,...,s_n)$. Let $a:=\text{max}\{s_i, r_i, 0: 1\leq i\leq n\}$ and for each $1\leq i\leq n$, consider nonnegative integers $k_i:=a-r_i$ and $l_i:=a-s_i$. Since $\chi_{(k_1,...k_n)}$ and $\chi_{(l_1,...l_n)}$ appear in  $\omega_{|_{\Toro^n}}$, then   $\chi_{(a,...,a)}$ appears in the decomposition of $\omega_{|_{\Toro^n}}\otimes \chi_{(r_1,...,r_n)}$ and in the decomposition of $\omega_{|_{\Toro^n}}\otimes \chi_{(s_1,...,s_n)}$, so $\chi_{(a,...,a)}$ appears at least twice in $\omega_{|_{\Toro^n}}\otimes\tau_{|_{\Toro^n}}$.
\end{itemize}

In conclusion, by Theorem \ref{torus}, we have proved  the following result.
\begin{proposition}
The triple $(\SO(2n), N(\mathfrak{so}(2n), \R^{2n}), \tau)$ is  commutative  if and only if  $\tau$ is the trivial representation.
\end{proposition}



\subsection*{Cases (V) and (VI)}

\quad  Let $n\geq 3$.
In case (V), $\mathfrak{n}=\mathfrak{su}(n)\oplus \C^n$, the representation $\pi$ of $\mathfrak{su}(n)$ on $\C^n$ is the canonical and then  $K=\SU(n)\times \Sphere^1$. 
In case (VI),  $\mathfrak{g}=\mathfrak{u}(n)$ is not semisimple and it is acting on $\C^n$ in the natural way. The subgroup of automorphisms $K$ is the same as in case (V), i.e., $K=\SU(n)\times \Sphere^1$.
\\ 

Every $\tau\in\widehat{K}$ is given by $\tau=\nu\otimes\chi_r$, for some $\nu\in\widehat{\SU(n)}$ and  some  character $\chi_r$ of $\Sphere^1$ with $r\in\mathbb{Z}$.

\begin{proposition}
The triple $(\SU(n)\times \Sphere^1, N(\mathfrak{su}(n),\mathbb{C}^n), \tau)$  is  commutative if and only if $\tau\in\widehat{\Sphere^1}$.
\end{proposition}

\begin{proof}
Let $\Toro^{n-1}$ be a maximal tours in $\SU(n)$.
In order to use Theorem \ref{torus}  we analyze how to decompose the metaplectic representation of $\Toro^{n-1}\times \Sphere^1$ acting on $\mathcal{P}(\C^n)$. We obtain
\begin{equation}
\omega_{|_{\Toro^{n-1}\times \Sphere^1}}=\bigoplus_{ m_1,...,m_n\in \mathbb{Z}_{\geq 0}}
\chi_{(m_1,...,m_n)},
\end{equation}
where $(\chi_{(m_1,...,m_n)}(\theta_1,...,\theta_n))(z_1^{m_1}...z_n^{m_n})=e^{-im_1\theta_1}z_1^{m_1}...e^{-im_n\theta_n}z_n^{m_n}$.
\\ \\
Let $\tau=\nu\otimes\chi_r$ as above. $\tau$  restricted to $\Toro^{n-1}\times \Sphere^1$ decomposes into a sum of characters of the form $\chi_{k_1,...,k_{n-1},r}$ for certain $k_1,...,k_{n-1}\in\mathbb{Z}$.
\\ \\
It is clear that if $\nu$ is the trivial representation of $\SU(n)$,  $(\omega\otimes\chi_r)_{|_{\Toro^{n-1}\times \Sphere^1}}$ decomposes into a sum of characters without multiplicity since
\begin{equation*}
(\omega\otimes\chi_r)_{|_{\Toro^{n-1}\times \Sphere^1}}=\bigoplus_{ m_1,...,m_{n-1},m_n\in \mathbb{Z}_{\geq 0}}
\chi_{(m_1,...,m_{n-1},m_n+r)}.
\end{equation*}
We suppose $(\nu, W_\nu)$ is  nontrivial, so $\dim(W_\nu)>1$. Let  $\chi_{k_1,...,k_{n-1}}$ and $\chi_{l_1,...,l_{n-1}}$ be two different characters that appear in $\nu_{|_{\Toro^{n-1}}}$. As a result, the characters \begin{equation*}
 \bigoplus_{ m_1,...,m_n\in \mathbb{Z}_{\geq 0}}\chi_{(m_1+k_1,...,m_{n-1}+k_{n-1},m_n+r)} \ \ \text{and} \ \ \bigoplus_{ \tilde{m}_1,...,\tilde{m}_n\in \mathbb{Z}_{\geq 0}}\chi_{(\tilde{m}_1+l_1,...,\tilde{m}_{n-1}+l_{n-1},\tilde{m}_n+r)}.
 \end{equation*}
appear in the decomposition of $(\omega\otimes\tau\otimes\chi_r)_{|_{\Toro^{n-1}\times \Sphere^1}}$. 
We obtain multiplicity in the decomposition of $(\omega\otimes\nu\otimes\chi_r)_{|_{\Toro^{n-1}\times \Sphere^1}}$ when we choose, for example, $m_n=\tilde{m}_n=0$ and for $1\leq i\leq n-1$:
\begin{equation}
     \left\{
	       \begin{array}{lll}
		 m_i=k_i-l_i  & \text{; } \tilde{m}_i=0  & \text{if } k_i-l_i \geq 0 \\ 
		 m_i=0 & \text{; } \tilde{m}_i=l_i-k_i & \text{if }
         k_i-l_i < 0  .        
	       \end{array}
	     \right.
\end{equation}
\end{proof}

\begin{proposition}
The triple $(\SU(n)\times \Sphere^1, N(\mathfrak{u}(n),\mathbb{C}^n), \tau)$ is  commutative if and only if $\tau\in\widehat{\Sphere^1}$.
\end{proposition}

\begin{proof}
This result holds from the previous theorem and Corollary \ref{Klambda con regulares}. 
\end{proof}


\subsection*{Case (VII)}

\quad In the item (VII) of the list $\mathfrak{g}=\mathfrak{u}(2)$ acts on $V=(\C^2)^k\oplus (\C^2)^n$ by $\pi$ defined in the following way:
\begin{itemize}
\item[$(i)$] $\mathfrak{u}(2)$ acts on each  of the $k$ components of $(\C^2)^k$ in the natural way,
\item[$(ii)$] in $(\C^2)^n$ the center of $\mathfrak{u}(2)$ acts trivially and the semisimple part acts as $\mathfrak{sp}(1)$ (or $Im(\HH)$) on the left side in each  of the $n$ components of $(\C^2)^n$.
\end{itemize}

Therefore the unitary group of intertwining operators is $U=\UU(k)\times \Sp(n)$. The nilpotent group $N(\mathfrak{u}(2),(\mathbb{C}^2)^k\oplus(\mathbb{C}^2)^n))$ has Lie algebra $\mathfrak{n}=\mathfrak{u}(2)\oplus (\C^2)^k\oplus (\C^2)^n$ and its group of orthogonal automorphisms is $K=\SU(2)\times \UU(k)\times \Sp(n)$. 
A maximal torus on $\SU(2)$ is given by $\Toro^1=\{\left(\begin{smallmatrix} e^{i\theta} & 0\\ 0 & e^{-i\theta} \end{smallmatrix}\right) | \ \theta\in\R\} $ and it is in correspondence with $\{e^{i\theta}| \ \theta\in\R \}$ in $\Sp(1)$.
\\ 

The metaplectic representation $\omega$ of $\Toro^1\times U$ acts on $\mathcal{P}((\C^2)^k\oplus (\C^2)^n)$ in the following way:
\begin{itemize}
\item[(i)] In order to understand the action of $\UU(k)$ on $\mathcal{P}((\C^{2})^k)$ we must think each element $A\in \UU(k)$ as an (unitary) intertwining operator of the action of $\mathfrak{u}(2)$ on $\mathcal{P}((\C^{2})^k)$: since the natural action of $\mathfrak{u}(2)$ on $\C^2$ is irreducible, Schur's lemma implies that $A$ is a linear operator  $A:(\C^2)^k\rightarrow (\C^2)^k$  whose matrix representation is
\begin{equation*}
[A]=\begin{pmatrix} 
a_{11}I & a_{12}I & ... & a_{1k}I\\ 
a_{21}I & a_{22}I & ... & a_{2k}I\\
... & ... & ... & ...\\
... & ... & ... & ...\\
a_{k1}I & a_{k2}I & ... & a_{kk}I
\end{pmatrix}, \textrm{ where } I=\begin{pmatrix} 
1 & 0\\ 
0 & 1
\end{pmatrix}, \text{ and } a_{i,j}\in \C.
\end{equation*}
From here the action of $\UU(k)$ on $\mathcal{P}((\C^2)^k)$  splits as $\mathcal{P}(\C^k)\otimes \mathcal{P}(\C^k)$ and also it can be written as $\bigoplus_{r,s\in \Z_{\geq 0}}\mathcal{P}_r(\C^k)\otimes \mathcal{P}_s(\C^k)$. Although $\mathcal{P}_r(\C^k)\otimes \mathcal{P}_s(\C^k)$ is not irreducible as $\UU(k)$-module, it is multiplicity free. Moreover, from $\cite{Yakimova}$,  $(\UU(n),N(\R,\C^n))$ is a strong Gelfand pair. In particular, for each $s\in\Z_{\geq 0}$, the sum $\bigoplus_{r\in \Z_{\geq 0}}\mathcal{P}_r(\C^k)\otimes \mathcal{P}_s(\C^k)$ is multiplicity free as $\UU(k)$-module. 

Apart from that, $\Toro^1\subset \SU(2)$ acts on each polynomial $p\in\mathcal{P}((\C^{2})^k)$ by
\begin{equation*}
p(u_1,v_1,...,u_k,v_k) \mapsto p(e^{i\theta}u_1,e^{-i\theta}v_1,...,e^{i\theta}u_k,e^{-i\theta}v_k), \quad (u_i, v_i\in\C).
\end{equation*}
Therefore the action of $\Toro^1\times \UU(k)$ on $\mathcal{P}((\C^2)^k)$ decomposes without multiplicity into
\begin{equation}\label{meta viii parte 1}
\bigoplus_{r, s\in \Z_{\geq 0}}\chi_{r-s}\otimes\upsilon_{(r)}\otimes\upsilon_{(s)},
\end{equation}
denoting by $\upsilon_{(r)}$ and $\upsilon_{(s)}$ the classical irreducible representations of $\UU(k)$ on $\mathcal{P}_r(\C^k)$ and $\mathcal{P}_s(\C^k)$ respectively and where $\chi_{r-s}(e^{i\theta})=e^{-(r-s)i\theta}$.

According to $\cite{Yakimova}$, for each  $s\in \Z_{\geq 0}$, 
\begin{equation*}
\bigoplus_{r\in \Z_{\geq 0}}\chi_{r}\otimes\upsilon_{(r)}\otimes\upsilon_{(s)}
\end{equation*}
is multiplicity free. Hence
\begin{equation*}
\bigoplus_{r, s\in \Z_{\geq 0}}\chi_{r-s}\otimes\upsilon_{(r)}\otimes\upsilon_{(s)}=\bigoplus_{s\in \Z_{\geq 0}}\chi_{-s}\otimes\left(\bigoplus_{r\in \Z_{\geq 0}}\chi_{r}\otimes\upsilon_{(r)}\otimes\upsilon_{(s)}\right)
\end{equation*}
is also multiplicity free.


\item[(ii)] For each $j\in\Z_{\geq 0}$ the action of $\Sp(n)$ on $\mathcal{P}_j(\C^{2n})$, denoted by $\eta_{(j)}$, is irreducible. $\Toro^1\subset \Sp(1)$ acts on each homogeneous polynomial by its degree, in the sense that if $p \in\mathcal{P}_j(\C^{2n})$ the action is $(e^{i\theta},p)\mapsto e^{-ji\theta}p$. Therefore the action of $\Toro^1\times \Sp(n)$ on $\mathcal{P}(\C^{2n})$ decomposes into irreducible representations as    
\begin{equation}
\bigoplus_{j\in\Z_{\geq 0}}\chi_j\otimes \eta_{(j)}.
\end{equation}
\end{itemize}

In conclusion, the metaplectic representation decomposes multiplicity free into
\begin{equation}
\omega_{|_{\Toro^1 \times U}}=\left(\bigoplus_{r, s\in \Z_{\geq 0}}\chi_{r-s}\otimes\upsilon_{(r)}\otimes\upsilon_{(s)}\right)\otimes \left(\bigoplus_{j\in\Z_{\geq 0}}\chi_j\otimes \eta_{(j)}\right)
\end{equation}
(but we must keep in mind that each term $\upsilon_{(r)}\otimes\upsilon_{(s)}$ is not irreducible).
\\ 

An irreducible unitary representation of $K$ is given by $\tau=\nu_d\otimes \upsilon \otimes \eta$, where $(\nu_d,\mathcal{P}_d(\C^{2}))$ is the well known irreducible unitary  representation of $\SU(2)$ on the homogeneous polynomials on $\C^{2}$ of degree $d$, $\upsilon\in\widehat{\UU(k)}$ and $\eta\in \widehat{\Sp(n)}$. This representation restricted to $\Toro^1\times U$ decomposes into irreducible representations as
\begin{equation*}
\tau_{|_{\Toro^1\times U}}=(\oplus_{i=0}^d\chi_{d-2i})\otimes\upsilon\otimes\eta.
\end{equation*}

Therefore $\omega\otimes\tau$ restricted to $\Toro^1\times U$ decomposes into
\begin{equation}\label{meta-tau VIII}
(\omega\otimes\tau)_{|_{\Toro^1\times U}}=\bigoplus_{r,s,j\in\Z_{\geq 0}}\oplus_{i=0}^d\chi_{r-s+j+d-2i}\otimes \upsilon_{(r)}\otimes \upsilon_{(s)}\otimes \upsilon \otimes\eta_{(j)}\otimes\eta.
\end{equation}

\begin{proposition}\label{teo ej viii}
The triple $(\SU(2)\times \UU(k)\times \Sp(n), N(\mathfrak{u}(n),(\mathbb{C}^2)^k\oplus(\mathbb{C}^2)^n), \tau)$ is commutative if and only if $\tau\in\widehat{\UU(k)}$.
\end{proposition}
\begin{proof}Let $\tau=\nu_d\otimes \upsilon \otimes \eta\in\widehat{\SU(2)\times \UU(k)\times \Sp(n)}$.
\\ \\
Let $\eta$ be nontrivial. From Corollary \ref{coro 1 sp}, $\eta$ appears in the decomposition into irreducible representations of $\eta_{(2)}\otimes\eta$. Thus, considering on the one hand $j=0$, $r=1$, $s=0$ and on the other hand $j=2$, $r=0$, $s=1$, we obtain repetition in (\ref{meta-tau VIII}).
\\ \\
Let $\nu_{d}$ be nontrivial, i.e., $d\geq 1$. We obtain repetition of representations in (\ref{meta-tau VIII}) considering, on the one hand $i=0$ and $j=0$ and on the other hand $i=1$ and $j=2$. 
\\ \\
In conclusion, if $\nu_d$ or $\eta$ are nontrivial, from Corollary \ref{Klambda con regulares}, we obtain  non-commutative triples. 
\\ \\
Assume $\nu_d$ and $\eta$ are trivial. We will show that the decomposition \begin{equation*}
(\omega\otimes \tau)_{|_{\Toro^1\times \UU(k)\times \Sp(n)}}=\bigoplus_{r, s, j\in\Z_{\geq 0}}\chi_{r-s+j}\otimes\upsilon_{(r)}\otimes \upsilon_{(s)}\otimes\upsilon\otimes \eta_{(j)}.
\end{equation*}
is multiplicity free by noting that its restriction to $\Toro^1\times Z(\UU(k))\times \Sp(n)$ is multiplicity free, where $Z(\UU(k))$ is the center of $\UU(k)$. Let $\chi_l$ be the character associated to $\upsilon_{|_{Z(\UU(k))}}$. 
 To each representation $\upsilon_{(r)}\in\widehat{\UU(k)}$ acting on the space of homogeneous polynomials of degree $r$, the character associated corresponds to the degree $r$, i.e., $(\upsilon_{(r)})_{|_{Z(\UU(k))}}=\chi_r$. Thus,
\begin{equation}\label{rest al centro}
(\omega\otimes\upsilon)_{|_{\Toro^1\times Z(\UU(k))\times \Sp(n)}}=
\bigoplus_{r, s, j\in\Z_{\geq 0}}
\chi_{r-s+j}\otimes \chi_{r+s+l}\otimes\eta_{(j)}.
\end{equation}
The parameter $l$ is fixed and we could assume that the parameter $j$ is also fixed because when $j$ runs on $\Z_{\geq 0}$ the factor $\eta_{(j)}$ changes. Therefore we obtain multiplicity in (\ref{rest al centro}) if there exist parameters $r, s, r', s'\in \Z_{\geq 0}$, $r\neq r'$, $s\neq s'$, such that 
\begin{equation*}
     \left\{
	       \begin{array}{ll}
		 r-s=r'-s' \\
		 r+s=r'+s'
	       \end{array}
	     \right.
\end{equation*}
and this is not possible.
\\ \\
In conclusion, if $\nu_m$ or $\eta$ are trivial, from Corollary \ref{Klambda con regulares}, we obtain  commutative triples.
\end{proof}


\subsection*{Case (VIII)}

\quad In this case the Lie algebra $\mathfrak{g}$ will not be  semisimple. Let $\alpha$ and $\beta$ be two nonnegative integers and let $\{m_i\}_{i=1}^\beta$ be integers greater or equal than 3.  We take $$\mathfrak{g}:=\mathfrak{su}(m_1)\oplus...\oplus\mathfrak{su}(m_\beta)\oplus\mathfrak{su}(2)\oplus...\oplus\mathfrak{su}(2)\oplus \mathfrak{c},$$ where $\mathfrak{c}$ is its center  and there are $\alpha$  copies of $\mathfrak{su}(2)$. The abelian component satisfies that $1 \leq \dim(\mathfrak{c})< \alpha+\beta$. Let us consider the vector spaces 
\begin{gather*}
V_1:=\C^{m_1}, \cdots, \ V_\beta:=\C^{m_\beta}, \  V_{\beta+1}:=\C^{2k_1+2n_1}, \cdots,  \ V_{\beta+\alpha}:=\C^{2k_\alpha+2n_\alpha}, 
\end{gather*}
where $k_j\geq 1$ and $n_j\geq 0$ for all $1\leq j\leq \alpha$ and let $V:=\bigoplus_{i=1}^{\beta+\alpha} V_i$. 
 The nilpotent group  $N(\mathfrak{g}, V)$ is given in the following way: 
 \begin{itemize}
 \item[-] For each $1\leq i\leq \beta+\alpha$, $\mathfrak{c}$ has a maximal subspace, denoted by $\mathfrak{c}_i$, of dimension one acting non- trivially on $V_i$.
 \item[-] For $1\leq i\leq \beta$,
 $\mathfrak{su}(m_i)\oplus \mathfrak{c}_i$  acts on  $V_i$ as in the case (VI). That is,  since $\mathfrak{su}(m_i)\oplus \mathfrak{c}_i$ is isomorphic to $\mathfrak{u}(m_i)$, it acts in the natural way on $\C^{m_i}$. Apart from that, $\mathfrak{su}(m_i)\oplus \mathfrak{c}_i$ acts  trivially on $V_j$ for all $j\neq i$.
 \item[-] For $\beta+1<i\leq \beta+\alpha $, $\mathfrak{su}(2)\oplus\mathfrak{c}_i$  acts on $V_i$ as in the case (VII). That is, it acts naturally as $\mathfrak{u}(2)$ on $(\C^2)^{k_i}$ and $\mathfrak{su}(2)$ acts naturally as $\mathfrak{sp}(1)$ on $(\HH)^{n_i}$ (which is isomorphic to $(\C^2)^{n_i}$ a real vector spaces). Also,   $\mathfrak{su}(2)\oplus\mathfrak{c}_i$ acts trivially on $V_j$ for all $j\neq i$. 
 \end{itemize}

    The connected group of unitary intertwining operators is $$U:=\Sphere^1\times...\times \Sphere^1\times \UU(k_1)\times \Sp(n_1)\times...\times \UU(k_\alpha)\times \Sp(n_\alpha),$$ where there are $\beta$ copies of $\Sphere^1$.
\\


Let us consider first the case where $\mathfrak{g}=\mathfrak{su}(m)\oplus\mathfrak{su}(2)\oplus \mathfrak{c}$ and $V=\C^m \oplus (\C^2)^k\oplus (\C^2)^n$, for some integers $m\geq 3$, $k\geq 1$ and $n\geq 0$.  Here we have $K=\SU(m)\times \SU(2)\times \Sphere^1\times \UU(k)\times \Sp(n)$. Let $\Toro^{m-1}$ be a maximal torus of $\SU(m)$ and $\Toro^1$ a maximal torus of $\SU(2)$.
\\ 

The metaplectic representation $\omega$ of $\Toro^{m-1}\times \Toro^1\times \Sphere^1\times \UU(k)\times \Sp(n)$ acts on $\mathcal{P}(\C^m\oplus\C^{2k}\oplus\C^{2n})$ in the following way,
\begin{itemize}
\item[$(i)$] $\Toro^1\times \UU(k)\times \Sp(n)$ acts as in the case (VII) on $\mathcal{P}(\C^{2k}\oplus \C^{2n})$ and
\item[$(ii)$] $\Toro^{m-1}\times \Sphere^1$ is an $m$-dimensional torus in $\UU(m)$ that acts on $\mathcal{P}(\C^m)$ in the natural way.
\end{itemize}

Thus, the representation of ${\Toro^{m-1}\times \Toro^1 \times \Sphere^1 \times \UU(k)\times \Sp(n)}$ decomposes into
\begin{equation}
\left(\bigoplus_{j_1,...j_m\in\Z_{\geq 0}}\chi_{(j_1,...,j_m)}\right)\otimes\left(\bigoplus_{r, s\in \Z_{\geq 0}}\chi_{r-s}\otimes\upsilon_{(r)}\otimes\upsilon_{(s)}\right)\otimes \left(\bigoplus_{j\in\Z_{\geq 0}}\chi_j\otimes \eta_{(j)}\right)
\end{equation}
where in the first factor of the tensor product we have a sum of characters of the form $\chi_{(j_1,...,j_m)}(\theta_1,...,\theta_m):=e^{-j_1i\theta_1}...e^{-j_mi\theta_m}$, whereas the second factor, $\left(\bigoplus_{r, s\in \Z_{\geq 0}}\chi_{r-s}\otimes\upsilon_{(r)}\otimes\upsilon_{(s)}\right)\otimes \left(\bigoplus_{j\in\Z_{\geq 0}}\chi_j\otimes \eta_{(j)}\right)$, is the same as in the case (VII). Remember that each term $\upsilon_{(r)}\otimes\upsilon_{(s)}$ is not irreducible, but anyway $\omega_{|_{\Toro^{m-1}\times \Toro^1 \times \Sphere^1 \times \UU(k)\times \Sp(n)}}$ is multiplicity free \cite{Yakimova}.
\begin{proposition}\label{teo 1 ej x}
The triple $(\SU(m)\times \SU(2)\times \Sphere^1\times \UU(k)\times \Sp(n), N(\mathfrak{su}(m)\oplus\mathfrak{su}(2)\oplus \mathfrak{c},\C^m \oplus (\C^2)^k\oplus (\C^2)^n, \tau)$ is  commutative if and only if $\tau\in\widehat{\Sphere^1\times \UU(k)}$.
\end{proposition}
\begin{proof}
An irreducible representation of $K$ is given by the tensor product $\tau=\nu\otimes\nu_d\otimes\chi_t\otimes \upsilon \otimes \eta$, where $\nu\in\widehat{\SU(m)}$, $(\nu_d,\mathcal{P}_d(\C^{2}))$ is the well known irreducible representation of $\SU(2)$ on the homogeneous polynomials on $\C^{2}$ of degree $d$, $\chi_t$ is a character of $\Sphere^1$ ($t\in\Z$), $\upsilon\in\widehat{\UU(k)}$ and $\eta\in \widehat{\Sp(n)}$.
\\ \\
First of all note that if $\nu_d$ or $\eta$ are nontrivial,  from Proposition \ref{teo ej viii},  we have multiplicity in the decomposition of $\omega\otimes\tau$ when we restrict it to  $\Toro^{m-1}\times \Toro^1 \times \Sphere^1 \times \UU(k)\times \Sp(n)$. Hence we suppose $\nu_d$ and $\eta$ trivial.
\\ \\
Now, let $\nu\in\widehat{\SU(m)}$ be nontrivial. 
From the theorem of the highest weight, $\nu $ can be parametrized in terms of a partition $(\nu_1,...,\nu_{m-1})$  
and every weight of $\nu$ is of the form
\begin{equation*}
\nu-a:=(\nu_1-a_1,\nu_2+a_1-a_2,...,\nu_{n-2}+a_{n-3}-a_{n-2},\nu_{m-1}+a_{m-1}),
\end{equation*}
where $a_i\in\Z_{\geq 0}$ for all $1\leq i\leq m-1$. 
Let $(\nu_1,...,\nu_{m-1})$ and $\nu-a$
be two different weights of the representation $\nu$ with $a_i\in\Z_{\geq 0}$ $\forall i$. 
In the decomposition of $\omega\otimes(\nu\otimes\chi_t\otimes\upsilon)$ restricted to $\Toro^{m-1}\times \Toro^1 \times \Sphere^1 \times \UU(k)\times \Sp(n)$ we have, in particular, the following terms
\begin{equation*}
\left(\bigoplus_{j_1,...j_m\in\Z_{\geq 0}}\chi_{\nu+(j_1,...,j_{m-1})}\otimes\chi_{j_m+t}\right)\otimes\left(\bigoplus_{r, s\in \Z_{\geq 0}}\chi_{r-s}\otimes\upsilon_{(r)}\otimes\upsilon_{(s)}\otimes \upsilon\right)\otimes \left(\bigoplus_{j\in\Z_{\geq 0}}\chi_j\otimes \eta_{(j)}\right)
\end{equation*}
and
\begin{equation*}
\left(\bigoplus_{j_1',...j_m'\in\Z_{\geq 0}}\chi_{\nu-a+(j_1',...,j_{m-1}')}\otimes\chi_{j_m'+t}\right)\otimes\left(\bigoplus_{r, s\in \Z_{\geq 0}}\chi_{r-s}\otimes\upsilon_{(r)}\otimes\upsilon_{(s)}\otimes \upsilon\right)\otimes \left(\bigoplus_{j\in\Z_{\geq 0}}\chi_j\otimes \eta_{(j)}\right).
\end{equation*}
We obtain multiplicity considering for example,
\begin{equation*}
     \left\{
	       \begin{array}{ll}
		 j_1=0, \ \ \ \ \ \ \ \ \ \ \ \ j_1'=a_1, \\
		 j_i=a_i-a_{i+1}, \ \ j_i'=0\ \ \ \ \ \ \ \ \ \ \ \ \ \ \ \ \  \text{ if } a_i-a_{i+1}\geq 0 \ \text{ for } 1<i<m-1,\\
         j_i=0, \ \ \ \ \ \ \ \ \ \ \ \ j_i'=-(a_i-a_{i+1})\ \ \text{ if } a_i-a_{i+1}< 0 \ \text{ for } 1<i<m-1, \\
         j_{m-1}=a_{m-1}, \ \ \ j_{m-1}'=0.
	       \end{array}
	     \right.
\end{equation*}
In conclusion, if $\nu_d$, $\eta$ or $\nu$ are nontrivial, from Corollary \ref{Klambda con regulares}, we obtain  non-commutative triples. This leads us to suppose that $\nu_d$, $\eta$ and also $\nu$ are trivial representations.
\\ \\
Finally, for any $\chi_t\in\widehat{\Sphere^1}$ and any $\upsilon\in \widehat{\UU(k)}$, we must analyze
\begin{equation}
\left(\bigoplus_{j_1,...j_m\in\Z_{\geq 0}}\chi_{(j_1,...,j_m+t)}\right)\otimes\left(\bigoplus_{r, s\in \Z_{\geq 0}}\chi_{r-s}\otimes\upsilon_{(r)}\otimes\upsilon_{(s)}\otimes \upsilon\right)\otimes \left(\bigoplus_{j\in\Z_{\geq 0}}\chi_j\otimes \eta_{(j)}\right).
\end{equation}
The sum $\bigoplus_{j_1,...j_m\in\Z_{\geq 0}}\chi_{(j_1,...,j_m+t)}$ is multiplicity free for all $t\in\Z$ and from Proposition \ref{teo ej viii}, the decomposition   $\left(\bigoplus_{r, s\in \Z_{\geq 0}}\chi_{r-s}\otimes\upsilon_{(r)}\otimes\upsilon_{(s)}\otimes \upsilon\right)\otimes \left(\bigoplus_{j\in\Z_{\geq 0}}\chi_j\otimes \eta_{(j)}\right)$ is multiplicity free for all $\upsilon\in \widehat{\UU(k)}$. Thus,  by Corollary \ref{Klambda con regulares}, we obtain  commutative triples when $\tau=\chi_t\otimes\upsilon$.
\end{proof}

\bigskip
With slight changes in the proof of the above proposition we can easily finish the case (VIII).

\begin{proposition}\label{teo 2 ej x}
 The triple $(K,N(\mathfrak{g},V),\tau)$ is commutative if and only if   $\tau=\chi_{t_1}\otimes...\otimes\chi_{t_\beta}\otimes \upsilon_1\otimes...\otimes\upsilon_\alpha$, where $\chi_{t_i}\in\widehat{\Sphere^1}$ ($t_i\in\Z$) for all $1\leq i\leq \beta$ and $\upsilon_j\in \widehat{\UU(k_j)}$ for all $1\leq j\leq \alpha$.
\end{proposition}


\subsection*{Summary}\label{section conclusion}

Putting together Proposition 1 to 9 we obtain the complete list of nontrivial commutative triples $(G\times U, N(\mathfrak{g},V),\tau)$ as in Theorem \ref{teo 1}:
\begin{itemize}
\item $\left(\SU(2)\times \Sp(n), N(\mathfrak{su}(2),(\mathbb{C}^2)^n),\tau\right)$, for all $\tau\in\widehat{\SU(2)}$ and for all $\tau\in\widehat{\Sp(n)}$ associated to a constant partition of length at most $n$, where $n\geq 1$,
\item $(\SU(n)\times \Sphere^1, N(\mathfrak{su}(n),\mathbb{C}^n),\tau)$, for all $\tau\in\widehat{\Sphere^1}$, where $n\geq 3$,
\item $(\SU(n)\times \Sphere^1, N(\mathfrak{u}(n),\mathbb{C}^n),\tau)$, for all $\tau\in\widehat{\Sphere^1}$, where $n\geq 3$,
\item $(\SU(2)\times \UU(k)\times \Sp(n), N(\mathfrak{u}(n),(\mathbb{C}^2)^k\oplus(\mathbb{C}^2)^n),\tau)$,   for all $\tau\in \widehat{\UU(k)}$, where $k\geq 1, n\geq 0$,
\item $(G\times U, N(\mathfrak{g}, V), \tau)$ where
\begin{align*}
& \mathfrak{g}=\mathfrak{su}(m_1)\oplus...\oplus\mathfrak{su}(m_\beta)\oplus\mathfrak{su}(2)\oplus...\oplus\mathfrak{su}(2)\oplus\mathfrak{c} \,
\text{ with } \mathfrak{c} \text{ an abelian component,}  
\\ & G=\SU(m_1)\times...\times \SU(m_\beta)\times \SU(2)\times...\times \SU(2) \ ,\\
& V= \C^{m_1}\oplus...\oplus \C^{m_\beta}\oplus\C^{2k_1+2n_1}\oplus...\oplus\C^{2k_\alpha+2n_\alpha} \text{ and }\\
& U=\Sphere^1\times...\times \Sphere^1\times \UU(k_1)\times \Sp(n_1)\times...\times \UU(k_\alpha)\times \Sp(n_\alpha) 
\end{align*}
for all $\tau\in\widehat{\Sphere^1}\otimes...\otimes \widehat{\Sphere^1}\otimes  \widehat{\UU(k_1)}\otimes...\otimes \widehat{\UU(k_\alpha)}$, where $m_j\geq 3$ for all $1\leq j\leq\beta$, $k_i\geq 1$, $n_i\geq 0$ for all $1\leq i\leq\alpha$ and
\item $(\UU(n), N(\R,\mathbb{C}^n),\tau)$, for all $\tau\in\widehat{\UU(n)}$, where $n\geq 1$ (proved by Yakimova in \cite{Yakimova}).
\end{itemize}

\begin{remark}
The Gelfand pair $(\UU(n), N(\R,\mathbb{C}^n))$ (where $N(\R,\mathbb{C}^n)$ is the Heisenberg group and $\UU(n)$ is its maximal group of orthogonal automorphisms) is the only strong Gelfand pair in this family. 
\end{remark}

\begin{remark}
Let $U$ be the connected component of the unitary intertwining operators of $(\pi,V)$. Assume that $U$ is a direct product of groups and $U(n)$ appears in $U$ for some $n\geq 1$.  For $\tau\in\widehat{U(n)}$ the triple $(G\times U, N(\mathfrak{g},V),\tau)$ is a commutative triple. (See cases (V) to (IX).)
\end{remark}

\begin{theorem}\label{main result} Let $N(\mathfrak{g},V)$ be indecomposable which has a square integrable representation. The nontrivial triple 
$(G\times U, N(\mathfrak{g},V),\tau)$ is commutative if and only if one of the following holds:
\begin{itemize}
\item[$(i)$] $U$ can be written as a direct product of groups $U=L\times M$ where $L:=\UU(n_1)\times ... \times \UU(n_m)$ (for some $m\geq 1$ and $n_i\geq 1$) and $\tau\in \widehat{L}$. 
\item[$(ii)$] $N(\mathfrak{g},V)$ is the H-type group and $\tau$ is as in Proposition \ref{teo ej i}.
\end{itemize}  
\end{theorem}




\begin{thebibliography}{X}

\bibitem[BJR]{BJR} C. Benson, J. Jenkins, and G. Ratcliff, \textsl{On Gelfand pairs associated with nilpotent Lie groups}, Trans. Amer. Math. Soc. \textbf{321}, 85--116 (1999).

\bibitem[BJR1]{BJR1} C. Benson, J. Jenkins, and G. Ratcliff, \textsl{The Orbit Method and Gelfand pairs, Associated with Nilpotent Lie Groups}, The Journal of Geometric Analysis. \textbf{9}, 569--582 (1990).


\bibitem[FH]{Fulton y Harris} W. Fulton, and J. Harris, \textsl{Representations Theory. A first course}. Springer-Verlag, New York (1991).








%

\bibitem[HLLS]{Howe} R. Howe, R. Lávička, Soo Teck Lee and V. Souček, \textsl{A reciprocity law and the skew Pieri rule for the symplectic group}, arXiv:1611.08473 (2016).





\bibitem[K]{Knapp} A. Knapp, \textsl{Lie Groups Beyond an Introduction}, second edition.  Progress in mathematics, Birkhäuser, Volume 140 (2002).

\bibitem[KT]{Koike} K. Koike and I. Terada, \textsl{Young-diagrammatic methods for the representation theory of the classical groups of type $B_n$, $C_n$, $D_n$}, Journal of Algebra, \textbf{107}, 466–511 (1987).




\bibitem[L]{Lauret} J. Lauret, \textsl{Gelfand pairs attached to representations of compact Lie groups}, Transformation Groups. Math. Soc. \textbf{5}, 307--324 (2000).

\bibitem[L1]{Lauret nil} J. Lauret, \textsl{Homogeneous nilmanifolds attached to
representations of compact Lie groups}, manuscripta math.  \textbf{99}, 287–309 (1999).


\bibitem[O]{Okada} S. Okada, \textsl{Pieri rules for classical groups and equinumeration between generalized oscillating tableaux and semistandard tableaux}, ArXiv (2016).








\bibitem[RS]{Fulvio} F. Ricci, and A. Samanta \textsl{Spherical analysis on homogeneous vector bundles}, arXiv:1604.07301 (2016). 



%


%
%










\bibitem[Wi]{Wil} E. Wilson, \textsl{Isometry groups on homogeneous nilmanifolds}, Geom. Dedicata \textbf{12}, 337--346 (1996). 

\bibitem[Wo]{Wolf} J. Wolf, \textsl{Harmonic analysis on commutative spaces}.  Mathematical
Surveys
and
Monographs, American Mathematical Society,
Volume 142 (2007).

\bibitem[Y]{Yakimova} O. Yakimova, \textsl{Principal Gelfand pairs}. Transform. Groups \textbf{11}, 305--335 (1975).

\end{thebibliography}
\end{document}